\documentclass[11pt]{article}
\usepackage[margin=1.1in]{geometry}              
\usepackage{graphicx}
\usepackage{amssymb}
\usepackage{amsmath}
\usepackage{epstopdf}
\usepackage{color}
\usepackage{xcolor}
\usepackage{txfonts}

\usepackage{multicol}
\usepackage{caption}
\usepackage{soul}

\usepackage[normalem]{ulem}
\usepackage{amsfonts}
\usepackage{algorithmic}

\usepackage{lipsum}

\DeclareGraphicsExtensions{.eps,.pdf,.jpeg,.png}

\newcommand{\R}{\mathbb{R}}

\DeclareMathOperator{\tr}{tr}

\newcommand{\NL}{N\ell}
\newcommand{\SL}{S\ell}

\newcommand{\eps}{\varepsilon}
\newcommand{\bigO}{\mathcal{O}}
\newcommand{\CM}{\mathcal{M}}			
\newcommand{\CMl}{\CM_{a-}}			    
\newcommand{\CMr}{\CM_{a+}}				
\newcommand{\CMc}{\CM_{r}}				
\newcommand{\fold}{\mathcal{L}}			
\newcommand{\Fl}{\fold_{-}}				
\newcommand{\Fr}{\fold_{+}}					

\newcommand{\gL}{g_{\text{L}}}
\newcommand{\gK}{g_{\text{K}}}
\newcommand{\gCa}{g_{\text{Ca}}}
\newcommand{\EL}{E_{\text{L}}}
\newcommand{\EK}{E_{\text{K}}}
\newcommand{\ECa}{E_{\text{Ca}}}

\newcommand{\SCM}{\mathcal{Z}}          
\newtheorem{remark}{Remark}
\newtheorem{proposition}{Proposition}
\newtheorem{proof}{Proof}

\newcommand{\PM}[1]{{\color{black}{#1}}}

\newcommand{\Rev}[1]{{\color{black}{#1}}}

\title{Spike-Adding Mechanisms in a Three-Timescale System: Insights from the FitzHugh-Nagumo Model with Periodic Forcing}

\author{Pake Melland\thanks{Department of Mathematics, Oregon Institute of Technology, OR, USA,  
  Pake.Melland@oit.edu.}
\; and\; Rodica Curtu\thanks{Department of  Mathematics, University of Iowa, IA, USA,
  rodica-curtu@uiowa.edu.}
\; and\; Zahra Aminzare\thanks{Department of  Mathematics, University of Iowa, IA, USA,
  zahra-aminzare@uiowa.edu.}
  }
\date{}
\begin{document}

\maketitle

\begin{abstract}
 In this work, we investigate the spike-adding mechanism in a class of three-di\-men\-sio\-nal fast-slow systems with three distinct timescales, inspired by the FitzHugh-Nagumo (FHN) model driven by periodic input. First, we numerically generate a \Rev{spike-adding} diagram for the FHN model by varying the frequency and amplitude of the input, revealing that as the frequency decreases and the amplitude increases, the number of spikes within each burst grows. \Rev{We demonstrate that a similar spike-adding structure occurs in the more realistic, periodically forced Morris–Lecar neuronal model.} Next, we apply methods from geometric singular perturbation theory to compute critical and super-critical manifolds of the fast-slow system. We use them to characterize the emergence of new burst-spikes in the FHN model, when the periodic forcing resembles a low frequency-band brain rhythm. We then describe how the uncovered spike-adding mechanism defines the boundaries that separate regions with different spike counts in the parameter space.
\end{abstract}


\medskip

\textbf{MSCcodes}
34C23, 
34C25, 
34E13, 
34D15, 
92B05, 
37N25. 

\bigskip

\section{Introduction}\label{Sec:Introduction}
A common firing pattern in neurons and other excitable cells is bursting, which is characterized by an active phase (a sequence of spikes) followed by a silent phase. Early work by Rinzel and collaborators \cite{
Rinzel1985Bursting,Rinzel1987Formal, RinzelErmentrout1998} pioneered the classification of bursting oscillations by fast/slow geometric methods, focusing on the bifurcation of the fast subsystem where the slow variables are treated as fixed parameters. Initially, they identified three classes of bursts: square-wave, parabolic, and elliptic. Later,  in his  classification scheme  \cite{Izhikevich2006Dynamical}, Izhikevich renamed these bursting types based on the bifurcations of the corresponding equilibrium and limit cycles. For example, square-wave, parabolic and elliptic bursting were renamed fold/homoclinic, circle/circle, and subHopf/fold cycle bursting, respectively.

Two key questions regarding bursts, which were not  addressed comprehensively by previous classifications, concern the transition from spiking to bursting and the process of adding spikes within a burst. While the former question has received considerable attention \Rev{\cite{ErmentroutKopell1986,FoxRotsteinNadim2014, Innocenti2007Dynamical,Nowacki2012Dynamical, Sherman2001,  Terman1992, XJWang1993,HanBi2012,Ma2022}}, the latter remained less explored. In particular,  the spike-adding mechanism in square-wave bursting has been  studied both numerically and analytically through fold/homoclinic bifurcations \cite{Barrio2024Dynamics, Barrio2020Spike-adding, Barrio2021Classification},  codimension one- and two- homoclinic bifurcations \cite{Channell2007Origin, Linaro2012Codimension}, and via canard explosions \cite{Carter2020Spike-Adding, Nowacki2012Dynamical, Penalva2024Dynamics}. In contrast, there has been limited research on spike-adding mechanisms of parabolic and elliptic bursts. One example is Desroches et al.'s paper \cite{DESROCHES2016Spike-adding} that analyzed parabolic bursting  in a system with 2-slow/2-fast dynamics. Their four-dimensional system had folded saddle canards that were shown to define the mechanism for burst spike-generation. 

\Rev{In this paper, we study the spike-adding structure of a class of bursting generated by a \textit{three-dimensional} system with \textit{three distinct timescales}.
While this bursting does not necessarily fall into any of the three main 
classification categories,  it does share some features with the elliptic bursters, and others with the parabolic bursters (see Discussion).} 
  As a toy-model for our analysis we consider the FitzHugh-Nagumo (FHN) system with the fast-timescale equation (``the voltage equation") driven by a low-frequency periodic input. First, we investigate numerically the dynamics of the periodically forced FHN equations as the amplitude and frequency of the forcing term vary. We find a transition in the system's dynamics from single-spike  to bursting as the input frequency decreases and its amplitude increases. \Rev{Next, by converting the forced FHN model to a three-dimensional system with three timescales (fast/slow/super-slow), we  focus on the slowly oscillating  force-drive,  where the system exhibits bursting, and describe the underlying spike-adding mechanism. 
To emphasize the biological relevance of our mathematical results, we then  demonstrate that a similar spike-adding structure occurs in the more realistic Morris–Lecar (ML) neuronal model. We choose parameters for the unforced ML model that produce two subcritical Hopf bifurcations and nearby saddle-node of limit cycles \cite{ermentrout2010foundations}, similar to the unforced FHN model.}
  
Other studies \cite{desroches2018SIADS, LetsonEtAl2017} have discussed a similar problem in three-dimensional systems with three time scales. In \cite{LetsonEtAl2017}, the spike-adding mechanism is attributed to the presence of a delayed Hopf bifurcation and a folded node singularity, while in \cite{desroches2018SIADS} it is due to a Hopf bifurcation and a folded saddle singularity. The dynamics analyzed in our paper are 
closely related to the setups in \cite{desroches2018SIADS, LetsonEtAl2017}. 
However, they distinguish themselves  by requiring
a combination of  the aforementioned mechanisms, with spike-adding due to both folded node and folded saddle singularities. \Rev{Furthermore, our study focuses on a parameter region where the interaction between the folded singularities and the existing delayed Hopf is not apparent. In contrast, the interplay between the dynamics at the folded node and the folded saddle influences the bursting structure.}
 
The main contributions of this paper are threefold. We provide a detailed theoretical and numerical exploration of how the intrinsic dynamics of a neuronal model could be finely tuned by controlling the frequency and/or the amplitude of the external input. Specifically, the dynamics of interest are  bursting patterns characterized by a given number of spikes within each burst. 
We also link our theoretical results to an explicit biological question: how patterns of spikes or bursts of action potentials recorded  from neurons {\it in vivo} correlate with underlying brain rhythms like delta (0.5-4~Hz), theta (4-8~Hz), alpha (8-12~Hz), beta (13-35~Hz) or gamma (35-70~Hz) \cite{buzsakiWatson2012,TsodyksEtAl1996} or otherwise, how electrical stimulation modulates the number of spikes, and/or organize them in bursts, in neurons \cite{AdamEtAl2022, YokoiEtAl2019}. 
\Rev{Lastly, we contribute new results to the general theory of bursting by identifying distinct dynamical mechanisms for burst patterns to emerge and gain or lose spikes  in three-dimensional systems with three timescales.}

The paper is organized as follows. In Sec.~\ref{Sec:Full_Model}, we  introduce the  model: a two-dimensional FHN system with periodic input. We analyze its dynamics numerically \Rev{by constructing spike-adding diagrams that show how the number of spikes per input cycle varies with the frequency and amplitude of the input. We also provide a biological motivation for our study.}
In Sec.~\ref{Sec:Fast_slow}, 
we reformulate the slowly-driven two-dimensional FHN model into a three-dimensional {\it fast/slow/super-slow} system with three distinct timescales, and we  investigate this system using geometric singular perturbation theory. 
We state our main results in  Sec.~\ref{Sec:Spike_adding}: we show theoretically how spikes are added to or removed from bursts by linking the system's trajectories to folded singularity equilibria. 
In Sec.~\ref{Sec:spike-count-boundary}, 
we describe how the spike-adding mechanisms identified in Sec.~\ref{Sec:Spike_adding} define  the boundaries in parameter space that separate regions with different spike counts.
\Rev{In Sec.~\ref{sec:ML}, we demonstrate that a similar spike-adding structure occurs in the more realistic, periodically forced Morris–Lecar neuronal model.} 
We conclude with a short discussion of our results and future directions in \Rev{Sec.~\ref{Sec:discussion}}.   Details of analytical calculations presented in this paper are included in the Appendix.

\section{FHN model with slowly varying force drive}\label{Sec:Full_Model}
Consider the FitzHugh-Nagumo model
\begin{subequations} \label{E:FHN-1}
\begin{align}
\dfrac{dx}{dt} &= x-\frac{1}{3} x^3 -y - a +E \sin(\omega t) \\
\dfrac{dy}{dt}  &= \eps (x - b y) 
\end{align}
\end{subequations}
with positive parameters $a$, $b$, $\eps$, and input function $I(t)=E\sin(\omega t)$. Typically $\eps$ takes very small values ($\eps \ll 1$), so we will fix it here to $\eps=0.08$.
We will also set parameters $a$, $b$ to values
  $a=0.875$, $b=0.8$, so the unforced FHN system (\ref{E:FHN-1}a)-(\ref{E:FHN-1}b) has a unique equilibrium point, i.e., when $E=0$ the system's nullclines  have a single intersection point. 
 The non-autonomous term $E\sin(\omega t)$ is  periodic      with positive amplitude $E$ and angular frequency $\omega$ (simply referred here as {\it ``frequency"}). 

For our choice of parameters $a$, $b$ and if we replace the input with a constant, $I(t) = E_0$, the FHN model \eqref{E:FHN-1}  exhibits two subcritical Hopf bifurcations. Furthermore, near each of these bifurcation points the system has a   saddle-node of limit cycles (SNLC) bifurcation (see e.g., \cite{davison2019mixed}). For $E_0$ values smaller than the first Hopf point at $E_{01}$, the equilibrium of the  system is stable and $x$ ({\it ``the voltage''}) does not spike. As $E_0$ increases past $E_{01}$, the equilibrium becomes unstable and the system's trajectories are attracted to a stable limit cycle, and $x$ starts spiking. When $E_0$ passes through the second Hopf point, $E_{02}$, the equilibrium point becomes stable once again, and $x$ becomes saturated. As shown in \cite{davison2019mixed}, the first SNLC bifurcation occurs for input amplitudes slightly smaller than $E_{01}$, at $E_{SN1}$, where a  stable and an unstable limit cycle emerge. Then, as $E_0$ approaches the subcritical Hopf point, the amplitude of the unstable limit cycle decays to zero.  On the other hand, the  stable limit cycle increases its amplitude as $E_0$ approaches the Hopf point $E_{01}$, and persists until  $E_0$ approaches the second SNLC bifurcation point at $E_{SN2}$. 
The parameter region $E_{SN1} < E_0 < E_{01}$ \Rev{is of order $\varepsilon$; within an exponentially small interval near $E_{SN1}$,} small amplitude limit cycles exist. The small amplitude limit cycles sharply transition to large amplitude relaxation oscillations in this narrow parameter space, a phenomenon known as a \textit{canard explosion} \Rev{\cite{Carter2020Spike-Adding, Nowacki2012Dynamical, Penalva2024Dynamics}}. 
Solutions that spend non-negligible time along repelling regions of the phase space \Rev{(i.e., near the repelling portion of slow manifolds)} before making fast transitions to stable regions \Rev{(the attracting portion of slow manifolds)} are called \textit{canard solutions}. 

Since the canard solutions of the FHN model with constant input exist only in a narrow parameter space, the solutions are not robust and  challenging to compute numerically. A similar result, but for $E_{02} < E_0 < E_{SN2}$ is obtained near the upper subcritical Hopf point $E_{02}$.

In contrast to the FHN model with constant input, which only exhibits resting, regular spiking, or saturating modes, the FHN model with periodic input may exhibit bursting for an appropriate range of amplitude $E$ and frequency $\omega$. Bursting dynamics are characterized by intervals of spiking activity separated by periods of rest.
In this work, we investigate the behavior of system \eqref{E:FHN-1} when $E$ varies  between $0$ and $1$, and  $\omega$ takes values between $0$ and $1.2$ (see Figure~\ref{fig:bifurcation_E_vs_w}), paying particular attention to the dynamics obtained when $\omega \ll 1$, of order $\bigO(\eps)$ or smaller (Figure~\ref{fig:bifurcation_E_vs_w_region_I}). 
 
We computed numerically a \Rev{spike-adding} diagram for the FHN model with periodic input,  and projected it on the parameter plane $(\omega, E)$; Figures~\ref{fig:bifurcation_E_vs_w}-\ref{fig:bifurcation_E_vs_w_region_I} and Appendix~\ref{Subsec:appendix-simulations} for simulation details. This \Rev{spike-adding} diagram  depicts the boundaries between regions where the system's bursting dynamics exhibit a different  ratio $m:n$, where $m$ is the number of large amplitude oscillations, or spikes in $x$, and $n$ is the number of input periods ($T = 2\pi/\omega$), as $\omega$  and $E$ vary. For example, in Figure~\ref{fig:bifurcation_E_vs_w}, region \PM{\textbf{A}},  the input frequency $\omega$ is very small ($\omega \ll \eps$, where $\eps = 0.08$) and the system admits bursting as a sequence of spikes occurring over each input period (Figure~\ref{fig:bifurcation_E_vs_w}; right panel \PM{\textbf{A}}). 
 In Figure~\ref{fig:bifurcation_E_vs_w},  region \PM{\textbf{B}},  $\omega$  falls in the range $\bigO(\eps)$ to $\bigO(\sqrt\eps)$ as $\sqrt\eps \approx 0.3$, and  the system exhibits tonic spiking,  with one spike  over each input period. Then, in region \PM{\textbf{C}}, where the input frequency $\omega$ is between  $\bigO(\sqrt\eps)$ and $\bigO(1)$, the system exhibits mixed mode oscillations:  a mixture of $m<n$ spikes and $n-m$ small amplitude oscillations occurs over $n$ input periods (Figure~\ref{fig:bifurcation_E_vs_w}, right panel \PM{\textbf{C}}:  the occurrence of one spike and one small amplitude oscillation over two input periods). 
  Lastly, in region \PM{\textbf{D}}, where $\omega$ is large  ($\omega=\bigO(1)$), no spikes occur and only  small amplitude oscillations may exist.  In summary, the FHN model with periodic drive exhibits busting dynamics, and  the number of spikes in each burst increases as $E$ increases and $\omega$ decreases (see also Figure~\ref{fig:bifurcation_E_vs_w_region_I}). 
\begin{figure}
    \centering
    \includegraphics[width=1\textwidth]{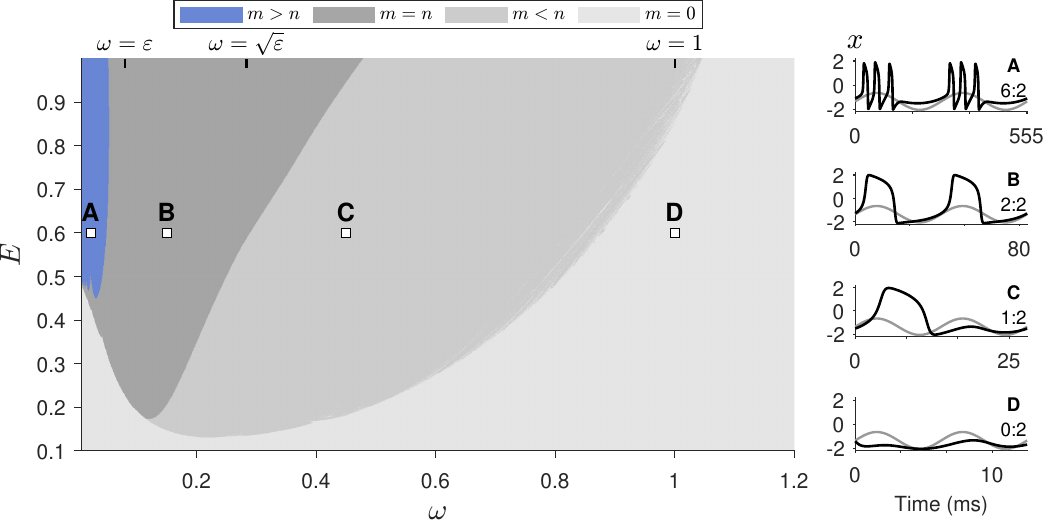}
    \caption{Left panel: The \Rev{spike-adding} diagram of FHN system \eqref{E:FHN-1} is shown as the frequency $\omega$ and amplitude $E$ of the periodic input vary. 
    Four regions (from left to right) indicate four different behaviors of the system in terms of its number of spikes ($m$) versus the number of input periods ($n$). 
    \PM{\textbf{A}} The system exhibits bursting, i.e., $m>n$. This region is plotted in more detail in Figure~\ref{fig:bifurcation_E_vs_w_region_I}; \PM{\textbf{B}}  The system exhibits tonic spikes, i.e.,  $m=n$;  \PM{\textbf{C}}  The system exhibits mixed mode oscillations, i.e., $m<n$;  \PM{\textbf{D}} No spikes occur ($m=0$). 
    Right panel: A trajectory of each region from the left panel is shown over two periods of the input ($n=2$) and $m=6, 2, 1$, and $0$, respectively. 
    The ticks at the top of the figure depict possible values  $\omega=\eps=0.08$, $\omega = \sqrt\eps\approx 0.3$, $\omega=1$,  given that $\eps=0.08$ in system \eqref{E:FHN-1}.}
    \label{fig:bifurcation_E_vs_w}
\end{figure}

In this paper, we extend the numerical results shown in Figures~\ref{fig:bifurcation_E_vs_w} and \ref{fig:bifurcation_E_vs_w_region_I}, to a theoretical study of the FHN model \eqref{E:FHN-1} with {\it slowly varying force drive}. That means, we will focus our analysis on  region \PM{\textbf{A}} from Figure~\ref{fig:bifurcation_E_vs_w} where the frequency of the input is very small, $\omega \ll \eps$,  say $\omega = \bigO(\eps \sqrt\eps)$ or lower. In this region, for a fixed value of $\omega$, over each input period $n=1$, the system could generate either no spikes ($m=0$), or 
one spike ($m=1$) or a burst with multiple spikes ($m>1$); see Figure~\ref{fig:bifurcation_E_vs_w_region_I}, regions \PM{\textbf{a}} through \PM{\textbf{d}}. 
\begin{figure}
    \centering
    \includegraphics[width=1\textwidth]{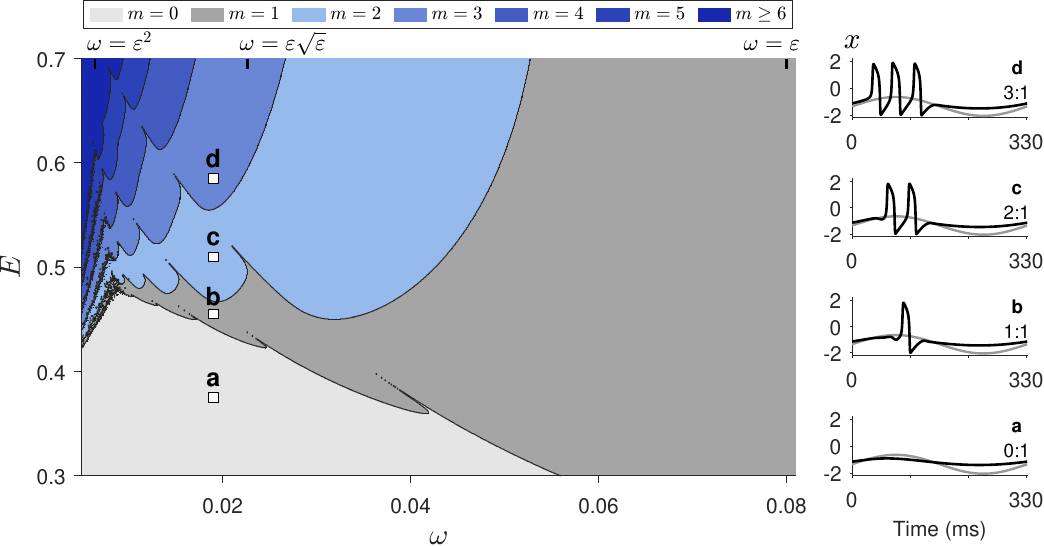} 
    \caption{Left panel: An enhancement of the blue region from Figure~\ref{fig:bifurcation_E_vs_w}, which {marks the spike-adding diagram of FHN system \eqref{E:FHN-1} as the frequency $\omega$ changes between 0 and $\eps$ and amplitude $E$ varies between 0.3 and 0.7}.  
    For each input period, the system can exhibit no spikes ($m=0$, light-gray), one spike ($m=1$, gray), or a burst with multiple spikes ($m>1$, diverse shades of blue). 
    The squares indicate parameter values for exemplary trajectories.
    Right panel: Solutions of \eqref{E:FHN-1}, with $(\omega,E)$ indicated by \PM{\textbf{a}--\textbf{d}} on the left panel, plotted over one period of the input ($n=1$). Shown here is an increasing sequence of $E$ values with $\omega$ fixed that yields increasing spike counts. Note that $\eps=0.08$ in system \eqref{E:FHN-1}. }
    \label{fig:bifurcation_E_vs_w_region_I}
\end{figure}
Our  objective is to characterize the spike adding mechanism in the emerging bursts as $\omega$ and $E$ vary. To achieve this goal we will employ in the next sections methods from Geometric Singular Perturbation Theory (GSPT) \cite{fenichel1979GSPT, guckenheimer2003forced, kuehn2015multiScaleGSPT} and the theory of folded singularities and canards \cite{brons2006MMO, CurtuRubin2011,  desroches2018SIADS, desroches2010numerical, szmolyan2001R3, wechselberger2005bifurcation,   wechselberger2013nonauto}.

\subsection{Biological relevance: spike-adding in bursts of neuronal activity driven by low-frequency brain rhythms or low-frequency electrical stimulation}

 \Rev{Brain rhythms across a wide range of frequency bands play critical roles in neural communication, cognitive processing, and brain health \cite{buzsaki2004neuronal, buzsaki2012rhythms}. Delta and theta oscillations are essential for sleep regulation and memory consolidation, particularly in the hippocampus, where theta rhythms coordinate encoding and retrieval processes \cite{herweg2020theta}. Alpha and beta rhythms are associated with attentional control and sensorimotor integration, while gamma oscillations are implicated in perception, working memory, and conscious awareness \cite{cho2023gamma, buzsaki2004neuronal}. High-gamma activity, often phase-locked to slower rhythms like theta, reflects localized population spiking and is closely tied to high-level cognitive functions \cite{canolty2006high}. Disruptions across these bands are linked to various neuropsychiatric disorders, making them valuable biomarkers and potential targets for therapeutic interventions \cite{basar2013brain}. Together, these rhythms form a dynamic framework that supports efficient information processing in the brain.
}

The FHN model is a two-dimensional \Rev{simplification} of the Hodgkin-Huxley model for neuronal activity that reproduces well the neuron's {\it excitability}  (activation/ deactivation dynamics), and the large excursions in the neuron's trajectory called {\it spikes} \cite{IzhikevichFitzHugh2006}; Figure~\ref{fig:bifurcation_E_vs_w_region_I}, right panels \PM{\textbf{b}--\textbf{d}}. The original equations depicting the dynamics of the neuron's membrane potential, $dV/dt = V - V^3/3 - W + I(t)$, and a recovery variable, $dW/dt = 0.08 (V + 0.7 - 0.8 W)$, are equivalent to system (\ref{E:FHN-1}) by taking $x=V$ and $y = W-a$ for $a=0.875$. The independent variable $t$ in (\ref{E:FHN-1}), while unitless, 
 measures the physical time in milliseconds. Therefore, the angular frequency in input $I(t) = E \sin(\omega t)$ gets calculated from the true frequency $f$ (in $\mbox{Hz}$, or $\mbox{s}^{-1}$)
as $\omega = 2 \pi f/1000$. Accordingly, diverse frequency bands for possible stimulation, or brain rhythms recorded simultaneously with neuron spike trains such as  
 delta (0.5-4~Hz), theta (4-8~Hz), alpha (8-12~Hz), beta (13-35~Hz), gamma (35-70~Hz) and high-gamma (70-200~Hz) would correspond to $\omega$ in the orders of magnitude shown in Figures~\ref{fig:bifurcation_E_vs_w}-\ref{fig:bifurcation_E_vs_w_region_I}. Specifically, frequencies $0.5$~Hz, $4$~Hz, $8$~Hz, $12$~Hz, $35$~Hz, $70$~Hz, $200$~Hz define $\omega\approx 0.003$, $0.025$, $0.05$, $0.075$ (Figure~\ref{fig:bifurcation_E_vs_w_region_I}) and $0.22$, $0.44$ and $1.26$ (Figure~\ref{fig:bifurcation_E_vs_w}) respectively. At $\eps=0.08$, these values show that our analysis has biological relevance. Indeed,   by taking $\omega$ in the parameter range $\bigO(\eps^2)-\bigO(\eps \sqrt\eps)-\bigO(\eps\sqrt[4]\eps)$
 we provide a mechanistic explanation on how spikes are added or subtracted from neuronal busting patterns, when they are stimulated in the low-frequency but physiological meaningful  bands delta and theta (Figure~\ref{fig:bifurcation_E_vs_w_region_I}, in different shades of blue). We investigate this problem in the next sections.
 \Rev{To further assess the biological relevance of our results, we also demonstrate that a similar spike-adding structure occurs in the more realistic, periodically forced Morris–Lecar neuronal model. See Section~\ref{sec:ML} for more details and comparison.} 


\section{Fast-slow analysis of the forced FHN model} \label{Sec:Fast_slow}
The \Rev{spike-adding} diagrams from Figures~\ref{fig:bifurcation_E_vs_w}-\ref{fig:bifurcation_E_vs_w_region_I} illustrate how slowly varying rhythms could organize spikes in neuronal bursting. The frequency of the input, $\omega$, is  independent from the internal properties of the neuron, e.g., from parameters $a$, $b$ and $\eps$. However, since $\omega$ itself is small, it could be defined as the product $\omega = \eps \delta$ with $\eps$ and $\delta$ taking values independently from each other. In this paper we adopt this (re)definition of $\omega$.

 Based on these observations, we analyze system (\ref{E:FHN-1}) in   two singular limits. We take the following two-step approach: 

First, in Sections~\ref{Subsec: Criticalmanifold} and \ref{Subsec: Foldpoints}, we fix $\delta>0$ and consider the fast-slow  analysis as $\varepsilon\rightarrow 0$. This limit induces a one-dimensional fast problem that admits a two-dimensional surface of equilibrium points. By changing to the $\varepsilon$-slow-time, the dynamics along hyperbolic regions of this surface are characterized by a two-dimensional reduced problem. GSPT breaks down along non-hyperbolic regions and induces robust canard dynamics, which contribute to a proposed {\it spike-adding canard-based mechanism} in the forced FHN system. 

Then, in Section~\ref{Subsec: SuperSlowManifold}, we fix $\varepsilon>0$ and take the limit 
 $\delta\rightarrow 0$. This limit establishes a two-dimensional fast problem that admits a curve of equilibrium points. Changing to the $(\varepsilon\delta)$-slow-time gives a one-dimensional reduced problem that characterizes the system's dynamics near hyperbolic  points on this curve. In this case, GSPT breaks down at non-hyperbolic points and induces  delayed-Hopf dynamics. \Rev{We discuss this case for completion purposes.} However, we note that   a {\it spike-adding delayed-Hopf-based} mechanism that\Rev{, in principle, may interact} with the aforementioned canard-induced mechanism \Rev{is not discernible in our parameter regime of interest, illustrated in Figure \ref{fig:bifurcation_E_vs_w_region_I}.} 

\begin{remark}
    Canard generation  has also been studied in  
 \cite{Bold2003forced, Burke2016Canards,guckenheimer2003forced}. 
When the system had two time scales, and the forcing term depended on low-frequency $\omega$ of order $\bigO(\eps)$, the authors showed that 
primary maximal canards induced by folded saddle-nodes of type I exist.
 Then, in the  intermediate-frequency regime 
  ($\omega = \bigO(\sqrt{\eps})$) and the high-frequency regime  ($\omega = \bigO(1)$), which involve three and two time scales respectively, the forced van der Pol system  was shown to exhibit torus canards. In contrast to these previous works, our paper focuses on the case where $\omega \ll \eps$, specifically the case of $\omega$ of order $\bigO(\eps\sqrt{\eps})$, $\bigO(\eps^2)$, or lower. This hypothesis  introduces three distinct time scales in the system, which produce spike-adding mechanisms through both folded saddle and folded node singularities.
\end{remark}
 
\subsection{Deriving the three-dimensional FHN  and the desingularized system} \label{S:derive-desingSyst} 

The time-dependency of the forced system \eqref{E:FHN-1} introduces a challenge when using analytic methods, as many theoretical statements from dynamical systems hold only for autonomous systems. So, we consider the augmented state space by defining the variable $\theta = \omega t$. 
We write the input frequency as $\omega =\varepsilon \delta$ where $\delta$ is chosen independently of $\varepsilon$. If $\delta = \mathcal{O}(1)$, then the forcing frequency is $\mathcal{O}(\varepsilon)$, and two timescales are present in the forced FHN system \eqref{E:FHN-1} - the voltage variable $x$ is the fast variable; the recovery variable $y$ and the forcing term are comparatively slow. As $\delta$ decreases from $\mathcal{O}(1)$ through $\mathcal{O}(\sqrt{\varepsilon})$ to $\mathcal{O}(\varepsilon)$, three timescales emerge, with the forcing term introducing a ``super-slow" dynamics. 

 Along with the change of variable $Y = y + a - E \sin(\omega t)$, system \eqref{E:FHN-1} becomes the three dimensional autonomous system:
\begin{subequations} \label{E:FHN-2}
\begin{align}
\dfrac{dx}{dt}  &= x-\frac{1}{3} x^3 - Y \label{E1:FHN-2}\\
\dfrac{dY}{dt}  &= \eps (x - b Y + a b - E b \sin \theta - E \delta \cos \theta) \label{E2:FHN-2}\\
\dfrac{d\theta}{dt}  &= \eps \delta\label{E3:FHN-2}.
\end{align}
\end{subequations}
Typically in fast-slow analysis, one needs to consider how equilibria of the original system interact with equilibria of the reduced system. Here, due to the constant dynamics in \eqref{E3:FHN-2}, the three dimensional system \eqref{E:FHN-2} admits no equilibrium points. Instead, we will consider the interaction of equilibria in the reduced system(s) obtained from taking the limits $\varepsilon\rightarrow0$ and $\delta \rightarrow 0$, separately.

For system \eqref{E:FHN-2}, the $x$-nullcline is the cubic function $Y = x -\frac{1}{3} x^3$ with a local minimum at $(x_L, Y_L) = (-1, -\frac{2}{3})$, which we will refer to as the \textit{left-knee},  and local maximum  at $(x_R, Y_R) = (1, \frac{2}{3})$, which we will refer to as the \textit{right-knee}. We shift the left-knee to the origin with the translation $u = x- x_L$, $v = Y - Y_L$ and obtain
\begin{subequations} \label{E:FHN-3}
\begin{align}
\dfrac{du}{dt}&= - v + F(u)\\
\dfrac{dv}{dt} &= \eps (u -b v + \mu - R_\delta \cos (\theta - \varphi_\delta)) \\
\dfrac{d\theta}{dt} &= \eps \delta \end{align}
\end{subequations}
where
\begin{eqnarray}
F(u) &=& u^2-\frac{1}{3} u^3 \label{E:F} \\
\mu &=& b\left(a+2/3\right) - 1 \approx 0.233  > 0\label{E:mu} \\
R_\delta &=& E \sqrt{b^2 + \delta^2}  \label{E:R}
\end{eqnarray}
and
\begin{subequations} \label{E:phi}
\begin{align}
&\cos (\varphi_\delta) = \frac{\delta}{\sqrt{b^2 + \delta^2}} \\
&\sin (\varphi_\delta) = \frac{b}{\sqrt{b^2 + \delta^2}} \label{E:sinphi}\\
&\varphi_\delta \in \left(0, \frac{\pi}{2}\right), \quad \varphi_\delta  \to \frac{\pi}{2} \,\, \mbox{as} \,\, \delta \to 0   \label{E:phi1}.
\end{align}
\end{subequations}

The subsequent analysis is performed using system \eqref{E:FHN-3} in variables $(u,v,\theta)$. However, we will represent all figures in terms of the original variables, $(x,y,\theta)$. 
In particular, the  fast variable $u$ is transformed to the original voltage variable $x$ with the shift $x = u + x_L=u-1$. Similarly, we write $y = v -2/3 - a + E \sin \theta$. Note that the two-dimensional surface derived from the limit $\eps\to0$  in system \eqref{E:FHN-3} has identical cross sections for fixed values of $\theta$; however,  this surface drawn in the original coordinates $x$, $y$ has cross sections that vary with $\theta$. Without loss of generality, we assume that $\theta\in[0,2\pi)$. 

\subsection{Critical manifold}
\label{Subsec: Criticalmanifold}
To perform a fast-slow analysis, we fix $\delta>0$ and take the singular limit of \eqref{E:FHN-3} as $\varepsilon\rightarrow 0$. In the limit, the variables $v$ and $\theta$ are constant and the associated \textit{fast  layer} problem of \eqref{E:FHN-3} is 
\begin{equation*}\label{E:Fast1}
\dfrac{du}{dt} = -v + F(u)
\end{equation*}
The set of equilibrium points for the fast layer problem defines the \textit{critical manifold},
\begin{equation*}
\CM = \{ (u,v,\theta) \in \R^2 \times [0, 2\pi) \,| \, \mathcal{F}(u,v) := - v + F(u) = 0\} \notag.
\end{equation*}
The critical manifold is a surface with attracting and repelling regions as determined by the eigenvalues of $F'(u)$. Attracting regions of $\CM$ occur where $\mathcal{F}_u(u,v) = $ $F'(u) < 0$, whereas regions with $\mathcal{F}_u(u,v) = $ $F'(u) > 0$ are repelling. 
Points in $\CM$ where $F'(u)=0$ are non-hyperbolic. Here, $F'(u)=-u(u-2)$, which motivates separating $\CM$ into attracting, repelling, and non-hyperbolic regions:
\begin{equation*}
    \CM  = \CMl \cup \Fl \cup \CMc\cup \Fr \cup \CMr \notag
\end{equation*}
where
\begin{eqnarray}
\CMl &=&  \{ (u,v,\theta) \,| \, \theta \in [0, 2\pi),  v = F(u) , u< 0 \} \notag \\
\Fl &=&  \{ (0,0,\theta) \, | \, \theta \in [0, 2\pi)  \} \notag \\
\CMc &=&  \{ (u,v,\theta) \,| \, \theta \in [0, 2\pi) ,   v = F(u) , 0 < u <2 \} \notag \\
\Fr &=&  \{ (2,\frac{4}{3}, \theta) \, | \, \theta \in [0, 2\pi) \} \notag \\
\CMr &=&  \{ (u,v,\theta) \, | \, \theta \in [0, 2\pi),  v = F(u) , u>2 \} \notag 
\end{eqnarray}
For reading purposes, we point out that the subscripts `$a$' and `$r$' refer to attracting and repelling regions, respectively. The signs `$+$' and `$-$' denote the signs of $u$ (or $x$)-values in each region - not the signs of $F'(u)$. 
Solutions to the fast layer problem are organized into ``layers" for fixed $v$ and $\theta$, so that solutions with initial conditions not on $\mathcal{M}$ make fast jumps towards attracting regions of $\CM$ and away from the repelling regions.
The critical manifold and it sub-regions are visualized using the original $x$, $y$, $\theta=\omega t$ coordinates in Figure~\ref{fig:CM_3D_2D}A.

Changing system (\ref{E:FHN-3}) to the slow time, $\tau_{1} = \varepsilon t$,  gives rise to the slow system,
\begin{subequations} \label{E:FHN-uvth-slow}
    \begin{align}
        \eps \frac{du}{d\tau_1} &= - v + F(u) \\
        \frac{dv}{d\tau_1} &= u -b v + \mu - R_\delta \cos (\theta - \varphi_\delta) \\
        \frac{d\theta}{d\tau_1} &= \delta.
    \end{align}
\end{subequations}
For the slow system, the limit $\varepsilon\to0$  gives the two-dimensional \textit{reduced slow} system,
\begin{subequations}\label{E:FHN-slow-reduced}
    \begin{align}
        \dfrac{dv}{d\tau_1} &= u-bv+\mu-R_{\delta}\cos\left(\theta-\varphi_{\delta}\right) \\
        \dfrac{d\theta}{d\tau_1} &= \delta ,
    \end{align}
\end{subequations}
which is accompanied by the algebraic equation $0 = -v+F(u)$. The algebraic equation recapitulates the critical manifold $\CM$ for $\theta\in[0,2\pi)$, and so the dynamics for the slow reduced problem \eqref{E:FHN-slow-reduced} occur along $\mathcal{M}$. 
Fenichel theory \cite{fenichel1979GSPT} states  that the dynamics on $\CM$ in the reduced slow system persist as slow dynamics for the full system \eqref{E:FHN-3}, with $0<\varepsilon\ll 1$,  along $\CM^{\varepsilon}$, a smooth perturbation of $\CM$. 
The stability properties of $\CM^{\varepsilon}$ are inherited from hyperbolic regions of $\CM$. 
To distinguish the perturbation $\CM^{\varepsilon}$ from the critical manifold $\CM$ derived from the singular limit, we will refer to $\CM^{\varepsilon}$ as the \textit{slow manifold}.
The distinction of the critical manifold $\CM$ from the perturbation $\CM^{\varepsilon}$ follows the convention described in \cite[Chapter 3]{kuehn2015multiScaleGSPT}.
For non-hyperbolic regions of $\CM$, Fenichel theory breaks down, possibly allowing for  richer dynamics in the original system near  these points. For $\CM$, the non-hyperbolic points occur along $\Fl$ and $\Fr$, called \textit{fold lines} (see dashed lines in Figure~\ref{fig:CM_3D_2D}A), where $F'(u)=0$.

Taking the derivative with respect to the slow-time of the algebraic equation $0=-v+F(u)$ that defines $\CM$,  gives $dv/d\tau_1 = F'(u)\times du/d\tau_1$. Thus, system \eqref{E:FHN-slow-reduced} is singular where $F'(u)=0$ (along fold lines) and is desingularized with the state-dependent timescale change $\tau_1 = -F'(u) \tau_D$,
\begin{equation} \label{E:FHN-desing}
    \begin{aligned}
        \dot{u} &=  R_\delta \cos (\theta - \varphi_\delta) - G(u) \\
        \dot{\theta} &= \delta u (u-2),
        \end{aligned}
\end{equation}
where dots denote derivatives with respect to $\tau_{D}$ and
\begin{equation} \label{E:fcG}
G(u) =  \mu + u - b F(u)  = \mu + u - b \left( u^2 - \frac{1}{3} u^3  \right)
\end{equation}
that satisfies $G'(u) = 1+bu(u-2) > 0$ for any real $u$ (i.e., $G(u)$ is strictly increasing) since $0<b<1$. The desingularized system \eqref{E:FHN-desing} retains the qualitative structure of the slow-system \eqref{E:FHN-slow-reduced}, with the exception that the vector field arrows are reversed on $\CMc$ where $-F'(u) < 0$.

We estimate the interval in which $u$ varies using the fast fibers of system (\ref{E:FHN-3}). Along $\Fl$, we have that $u=0$ and $F(0)=0$. 
So, a trajectory that makes a fast jump at $\Fl$ will again intersect $\CM$ when the cubic function satisfies $F(u)=0$, which occurs when $u=3$. 
Similarly, a trajectory that makes a fast jump at $\Fr$, exits when $u=2$, and intersects $\CM$ when $F(u)=F(2)=4/3$. 
This intersection occurs at $u=-1$. 
With these values, we approximate the range for $u$ as $u\in[u_{\text{min}},u_{\text{max}}] = [-1,3]$, or in the original coordinates, $x\in[x_{\text{min}},x_{\text{max}}]=[-2,2]$.

\subsection{Fold points and their stability analysis}
\label{Subsec: Foldpoints}
Due to the state-dependent timescale transformation that defines the desingularized system, all equilibrium points of \eqref{E:FHN-desing} occur along the fold lines and are called \textit{folded equilibria.} 
Equilibrium points of the desingularized system (\ref{E:FHN-desing}) satisfy $u=0$, $R_\delta \cos (\theta - \varphi_\delta) = G(0)$ and $u=2$, $R_\delta \cos (\theta - \varphi_\delta) = G(2)$, respectively.
Note that both $G(0)$ and $G(2)$ are positive ($G(0) = \mu$ and $G(2) = \mu + 2 - 4b/3 = \mu +\frac{4}{3}\left(\frac{3}{2} -b \right)>0$ since $\mu>0$, $0<b<1$). 

To maintain consistency across all figures, we ensure that all equilibria are chosen such that $\theta \in[0,2\pi)$. We then obtain the following result:

\begin{proposition}[Existence of folded equilibria] \label{Th:FoldedEq}
Consider $G$, $\mu$ from \eqref{E:fcG}, \eqref{E:mu},  $R_\delta$, $\varphi_\delta$ defined by \eqref{E:R}, \eqref{E:phi}, and parameters $a=0.875$, $b=0.8$  
 as in Sec.~\ref{Sec:Full_Model}.
Then:

i) System \eqref{E:FHN-3} admits 
two folded singularities on the left knee curve $\Fl$ of $\CM$ if and only if the amplitude $E$ of the  FHN input signal is large enough, i.e., 
\begin{equation} \label{E:saddle-node-left}
 E^{*}_{\ell,\delta} = \frac{ab-(1 - 2b/3)}{\sqrt{b^2 + \delta^2}} < E.
\end{equation}
These folded singularities, denoted by $P_{S\ell}$ and $P_{N\ell}$, are given by
\begin{eqnarray}
u_{S\ell} &=& 0, \quad  v_{S\ell} = 0, \quad \theta_{S\ell} = \varphi_\delta + \arccos \frac{G(0)}{R_\delta} \label{E:saddle} \\
u_{N\ell} &=& 0, \quad  v_{N\ell} = 0, \quad \theta_{N\ell} = \varphi_\delta - \arccos \frac{G(0)}{R_\delta}
+2\pi\quad (\text{mod}\; 2\pi).
\label{E:node}
\end{eqnarray}

ii) System \eqref{E:FHN-3} admits two folded singularities on the right knee curve $\Fr$ of $\CM$ if and only if the amplitude $E$ of the FHN  input signal is larger, i.e., 
\begin{equation} 
 E^{*}_{r,\delta} = E^{*}_{\ell,\delta} + \frac{2-4b/3}{\sqrt{b^2 + \delta^2}} < E.
 \label{E:saddle-node-right}
\end{equation}
These folded singularities are given by
\begin{eqnarray}
u_{Sr} &=& 2, \quad v_{Sr} = \frac{4}{3}, \quad \theta_{Sr} = \varphi_\delta - \arccos \frac{G(2)}{R_\delta} +2\pi\quad (\text{mod}\; 2\pi)\label{E:saddle-right} \\
u_{Nr} &=& 2, \quad v_{Nr} = \frac{4}{3}, \quad \theta_{Nr} = \varphi_\delta + \arccos \frac{G(2)}{R_\delta}. \label{E:node-right} 
\end{eqnarray}

iii) System \eqref{E:FHN-3}  has a pair of folded singularities on the left knee curve $\Fl$ but no folded singularities on $\Fr$ if the amplitude $E$ of the FHN input satisfies the condition
\begin{equation} \label{E:condE}
E^{*}_{\ell,\delta}  < E < E^{*}_{r,\delta}  \quad \Longleftrightarrow \quad 
\frac{ab-(1 - 2b/3)}{\sqrt{b^2 + \delta^2}} < E < \frac{ab+(1-2b/3)}{\sqrt{b^2 + \delta^2}}. 
\end{equation}
\end{proposition}

\begin{proof}
i) By the definition of $\Fl$, an equilibrium of \eqref{E:FHN-3} on $\Fl$ must satisfy $u=v=0$, and hence, $\mu - R_\delta \cos(\theta-\varphi_\delta)=0$, i.e., $\theta = \varphi_\delta \pm\arccos(\mu/R_\delta) = \varphi_\delta \pm\arccos(G(0)/R_\delta)$.
We note that $G(0)=\mu > 0$ and $R_\delta$ depends on $E$ with $\mu$, $R_\delta$ defined by \eqref{E:mu}, \eqref{E:R}. Under these constraints, the singularities on the left knee curve $\Fl$ exist and are distinct if and only if $G(0)/ R_\delta < 1$, or $0< \mu / (E\sqrt{b^2 + \delta^2}) < 1$ which is exactly the inequality \eqref{E:saddle-node-left}. Obviously, the ratio $G(0)/R_\delta$ could be $1$, but then $\theta_{Sl} = \theta_{Nl}$ that corresponds to a unique equilibrium (a folded saddle-node) on $\Fl$. We exclude this case from consideration. 

 Since  $G(0)/R_\delta>0$, we have $\arccos(G(0)/R_\delta)\in(0,\pi/2)$.
Coupling these bounds with the fact that $\varphi_{\delta}\in(0,\pi/2)$ (see property \eqref{E:phi1}), we  find a possible range for $\theta_{S\ell}$,  $\theta_{S\ell}\in(0,\pi)$. Thus,  no modification is required for the coordinates of this equilibrium.
On the other hand, $\varphi_{\delta} - \arccos(G(0)/R_\delta)\in(-\pi/2,\pi/2)$, so we may need to  modify the $\theta$-coordinate for this equilibrium by calculating its associated principal angle. This shift will bring the coordinate inside interval $[0,2\pi)$, which means that  either $\theta_{N\ell}\in(0,\pi/2)$ or $\theta_{N\ell}\in (3\pi/2,2\pi)$. 

ii) Similarly, by the definition of $\Fr$, an equilibrium of \eqref{E:FHN-3} on $\Fr$ must satisfy $u=2$, $v=4/3$, and hence, $G(2) - R_\delta \cos(\theta-\varphi_\delta)=0$, i.e., $\theta = \varphi_\delta \pm\arccos(G(2)/R_\delta)$.
Given that $G(2)>0$,  the singularities on the right knee curve $\Fr$ exist and are distinct if and only if $G(2)/R_\delta < 1$, i.e.,  
$0< (\mu +2 - 4b/3)/\sqrt{b^2 + \delta^2} < E$ that is equivalent to \eqref{E:saddle-node-right}.

For the equilibria on $\Fr$, we modify  $\theta_{Sr}$ eventually, as above, and find that either $\theta_{Sr}\in(0,\pi/2)$ or $\theta_{Sr}\in (3\pi/2,2\pi)$. On the other hand $\theta_{Nr}\in(0,\pi)$ by default. 

iii) Last statement is true because $E^{*}_{r,\delta}$ is always greater than $E^{*}_{\ell,\delta}$ for $0<b<1$.  
\end{proof}

\begin{figure}
    \centering    
    \includegraphics[scale=1]{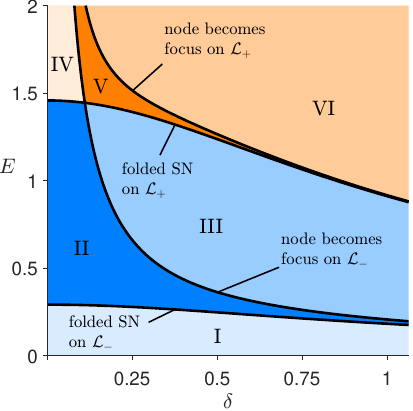}
    \caption{
    \textbf{Bifurcation diagram for the three-dimensional FHN.}
    Parameter regions according to folded equilibria:
    The blue regions (I, II, and III) show the primary parameter region considered in this paper for spike-adding. In regions I, II, III the FHN model admits folded equilibria only on $\Fl$. 
    The lines separating regions I from II and II from III are $E=E^{*}_{\ell,\delta}$ and $E=E^{**}_{\ell,\delta}$, with $E^{*}_{\ell,\delta}$, $E^{**}_{\ell,\delta}$ given by \eqref{E:saddle-node-left}, \eqref{E:Enode} respectively.   The orange regions (IV, V, and VI)  admit additional folded equilibria that lie on $\Fr$. 
    The lines separating regions III from V, and V from VI, are  $E=E^{*}_{r,\delta}$ and $E=E^{**}_{r,\delta}$ with $E^{*}_{r,\delta}$, $E^{**}_{r,\delta}$  given by Eqs.~\eqref{E:saddle-node-right} and \eqref{E:Enode-right}. 
    \Rev{Note that the horizontal axis closely resembles that of Figure~\ref{fig:bifurcation_E_vs_w_region_I} where $\omega=\varepsilon\delta\in[0,0.08]$ and $\varepsilon=0.08$.} }
  \label{F:ROI}
\end{figure}

\begin{remark} \label{R:bifdiag}
 According to formula \eqref{E:saddle-node-left}, the lower bound $E^{*}_{\ell,\delta}$ for $E$ increases when $\delta$ decreases. Furthermore, its limiting value can be estimated as $\delta \to 0$ and it is  $E^*_{\ell,0} = a+\frac{2}{3} - \frac{1}{b} \approx 0.2917 >0$. In the parameter plane $(\Rev{\delta},E)$, 
 the curve $E=E^{*}_{\ell,\delta}$ defines a folded saddle-node bifurcation on $\Fl$  (Figure~\ref{F:ROI}, this curve separates region I from II). In a similar manner, the curve $E=E^{*}_{r,\delta}$ defines a folded saddle-node bifurcation on $\Fr$  and $E^{*}_{r,\delta}$ increases as $\delta \to 0$ with limiting value $E^{*}_{r,0} = E^*_{\ell,0} + 2/b-4/3 \approx 1.4583$ (Figure~\ref{F:ROI}, the curve separates regions III and V, and region II from IV). Clearly, for parameter values below this curve (i.e., the blue-colored parameter range in Figure~\ref{F:ROI}) folded equilibria exist only on $\Fl$. 
\end{remark}

\begin{proposition}[Type and stability of folded equilibria] \label{Th:saddles-nodes}
Consider the hypotheses from Proposition~\ref{Th:FoldedEq}. Then: 

i) The folded equilibria $(u_{S\ell}, v_{S\ell}, \theta_{S\ell})$ on $\Fl$  and $(u_{Sr}, v_{Sr}, \theta_{Sr})$ on $\Fr$ (when they exist) are folded saddles. 

ii) The singularity $(u_{N\ell}, v_{N\ell}, \theta_{N\ell})$ on $\Fl$ (when it exists) is a folded node if 
\begin{equation} \label{E:Enode}
 E < E_{\ell,\delta}^{**} = \sqrt{ \left( E^{*}_{\ell,\delta} \right)^2 +    \frac{1/(8\delta)^2}{b^2 + \delta^2}}
\end{equation}
 and it is a folded focus if $E > E_{\ell,\delta}^{**}$.

iii) The folded equilibrium $(u_{Nr}, v_{Nr}, \theta_{Nr})$ on $\Fr$ (when it exists) is a folded node if 
\begin{equation} \label{E:Enode-right}
  E < E_{r,\delta}^{**} = \sqrt{ \left(E^{*}_{r,\delta}\right)^2 +    \frac{1/(8\delta)^2}{b^2 + \delta^2}}
\end{equation}
 and it is a folded focus  if $E > E_{r,\delta}^{**}$.
\end{proposition}
\begin{proof}
The linearization of (\ref{E:FHN-desing}) near any equilibrium $(u^*, \theta^*)$ gives the matrix
\[
J_{u\theta} = \left[ \begin{matrix} -G'(u^*)  & -R_\delta \sin (\theta^* - \varphi_\delta)  \\
2 \delta (u^*-1) & 0 \end{matrix} \right]
\]
which has trace and determinant given by
\[
\tr\left( J_{u\theta} \right) = -G'(u^*) = -1 
\]
and 
\[
\det \left( J_{u\theta}  \right) = 2 \delta (u^*-1) R_\delta \sin (\theta^* - \varphi_\delta) = \pm 2 \delta (u^*-1) \sqrt{R_\delta^2 - G(u^*)^2}\;.
\]
The trace is clearly negative for any equilibrium. For $\theta_{N\ell}$ and $\theta_{Sr}$, we have that $\sin(\theta^* - \varphi_{\delta}) = -\sin[\arccos(G(u^*)/R_\delta)] < 0$; see  \eqref{E:node}-\eqref{E:saddle-right} with $G(u^*)>0$. Since $u^*=0<1$ on $\Fl$ and  $u^*=2>1$ on $\Fr$, the determinant $\det\left(J_{u\theta}\right)$ is positive for $\theta_{N\ell}$ (thus the folded equilibrium is asymptotically stable) and negative for $\theta_{Sr}$ (a \Rev{folded} saddle). A similar calculation shows that $\sin(\theta^*-\varphi_\delta) =  \sin[\arccos(G(u^*)/R_\delta)] > 0$ for $\theta_{S\ell}$ (where $u^*=1$ thus a \Rev{folded} saddle) and $\theta_{Nr}$ (where $u^*=2$ thus an asymptotically stable equilibrium). We note that the equilibrium is a {\it folded node} only when $1 - 4 \det\left( J_{u\theta} \right) > 0$, which at the equilibrium is equivalent to  $1/(8\delta)^2  + G(u^*)^2 > R_\delta^2$. By direct replacement of $u^*=0$ on $\Fl$ and $u^*=2$ on $\Fr$, and using the definition of $G$, $\mu$ from \eqref{E:fcG}, \eqref{E:mu}  we obtain the results in (ii) and (iii). 
\end{proof}

\begin{remark}
The curve $E=E_{\ell,\delta}^{**}$ with $E_{\ell,\delta}^{**}$ defined by \eqref{E:Enode} separates regions   II from III and IV from V in Figure~\ref{F:ROI}. From \eqref{E:Enode} we observe that $E_{\ell,\delta}^{**} \to \infty$ as $\delta \to 0$, so this curve intersects $E=E_{r,\delta}^{*}$ (i.e., the curve that separates region III from region V) for some small value of $\delta$. In particular, at $a=0.875$, $b=0.8$ we find the intersection of these curves at $\delta \approx 0.0957$, i.e., at $\omega \approx 0.008 = \bigO(\eps^2)$. For larger values than $\delta \approx 0.0957$ the curve marking the transition on $\Fl$ from a folded node to a folded focus lies below the curve of saddle-node bifurcation on $\Fr$. Another observation is that at $\delta=1$ (or, $\omega=\eps$) we get $E=E^{**}_{\ell,1} \approx 0.2067$ approaching very closely $E=E^{*}_{\ell,1} \approx 0.1822$.  Finally, the curve $E=E_{r,\delta}^{**}$ with $E_{r,\delta}^{**}$ defined by  \eqref{E:Enode-right} separates regions V and VI in Figure~\ref{F:ROI}. We find that $E_{r,\delta}^{**} \to \infty$ as $\delta \to 0$ and that the curve $E=E_{r,\delta}^{**}$ approaches curve $E=E_{r,\delta}^{*}$ when $\delta \to 1$, with $E_{r,1}^{**} \approx 0.9162$ near  $E_{r,1}^{*} \approx 0.9110$.
\end{remark}

\subsubsection*{Eigenvalues and eigenvectors of folded singularities}
System \eqref{E:FHN-3}'s dynamics  on the critical manifold $\CM$ (and, by perturbation, on the slow manifold $\CM_\eps$) are highly influenced
by the  attractive/repelling manifolds of the folded nodes and folded saddles. Near the fold points on $\Fl$ and $\Fr$, the directions of attraction (for the \Rev{folded} nodes) or of attraction and repulsion (for the \Rev{folded} saddles) are well approximated  by the corresponding eigenvectors. We calculate these quantities below.

\begin{proposition} \label{Th:eig}
Consider the folded saddles and folded nodes (when they exist) from Proposition~\ref{Th:saddles-nodes}, and $\mu$ defined by \eqref{E:mu}. Then, in the desingularized system~\eqref{E:FHN-desing}

i) the folded saddle 
on $\Fl$ has coordinates $(u_{S\ell}, \theta_{S\ell})$ with eigenvalues 
\begin{eqnarray}
\lambda_{S\ell,1} &=& -\frac{1}{2} - \frac{1}{2} \sqrt{1+ 8 \delta 
\left[ \, E^2 b^2 - \mu^2 + E^2 \delta^2 \, \right]^{1/2}}   \notag \\
\lambda_{S\ell,2} &=& -\frac{1}{2} + \frac{1}{2}\sqrt{1+ 8 \delta 
\left[ \, E^2 b^2 - \mu^2 + E^2 \delta^2 \, \right]^{1/2} }\,,  \notag 
\end{eqnarray}
and corresponding eigenvectors
\[
\vec{V}_{S\ell,1} = (\lambda_{S\ell,1} \, , \, - 2\delta)^\top, \quad
\vec{V}_{S\ell,2} = (\lambda_{S\ell,2} \, , \,  - 2\delta)^\top.\]

ii) the folded node 
on $\Fl$ has coordinates $(u_{N\ell}, \theta_{N\ell})$ with eigenvalues 
\begin{eqnarray}
\lambda_{N\ell,1} &=& - \frac{1}{2}  - \frac{1}{2}\sqrt{1- 8 \delta
\left[ \, E^2 b^2 - \mu^2 + E^2 \delta^2 \, \right]^{1/2}}  \notag \\
\lambda_{N\ell,2} &=& -\frac{1}{2} + \frac{1}{2}\sqrt{1- 8 \delta 
\left[ \, E^2 b^2 - \mu^2 + E^2 \delta^2 \, \right]^{1/2}}   \notag 
\end{eqnarray}
and corresponding eigenvectors
\[\vec{V}_{N\ell,1} = (\lambda_{N\ell,1} \,, \, - 2\delta)^\top , \quad 
\vec{V}_{N\ell,2} = (\lambda_{N\ell,2}, - 2\delta)^\top.\]

iii) the folded saddle 
on $\Fr$ has coordinates $(u_{Sr}, \theta_{Sr})$ with eigenvalues 
\begin{eqnarray}
\lambda_{Sr,1} &=& -\frac{1}{2} - \frac{1}{2} \sqrt{1+ 8 \delta 
 \left[ \, E^2 b^2 - (\mu + 2 -4b/3)^2  + E^2 \delta^2 \, \right]^{1/2} }   \notag \\
\lambda_{Sr,2} &=& -\frac{1}{2} + \frac{1}{2} \sqrt{1+ 8 \delta \left[ \, E^2 b^2  - (\mu + 2 -4b/3)^2 + E^2 \delta^2 \, \right]^{1/2} }  \notag 
\end{eqnarray}
and corresponding eigenvectors 
\[
\vec{V}_{Sr,1} = (\lambda_{Sr,1}\, , \, 2\delta)^\top, \qquad
\vec{V}_{Sr,2} = (\lambda_{Sr,2} \,, \, 2\delta)^\top.
\]

iv) the folded node 
on $\Fr$ has coordinates 
$(u_{Nr}, \theta_{Nr})$ with eigenvalues
\begin{eqnarray}
\lambda_{Nr,1} &=& -\frac{1}{2} - \frac{1}{2} \sqrt{1- 8 \delta
 \left[ \, E^2 b^2  - (\mu + 2 -4b/3)^2 + E^2 \delta^2 \, \right]^{1/2}}    \notag \\
\lambda_{Nr,2} &=& -\frac{1}{2} + \frac{1}{2} \sqrt{1- 8 \delta \left[ \, E^2 b^2  - (\mu + 2 -4b/3)^2 + E^2 \delta^2 \, \right]^{1/2}}   \notag 
\end{eqnarray}
and corresponding eigenvectors 
\[
\vec{V}_{Nr,1} = (\lambda_{Nr,1} \,, \, 2\delta)^\top , \qquad
\vec{V}_{Nr,2} = (\lambda_{Nr,2} \,, \, 2\delta)^\top.
\]
\end{proposition}
\begin{proof}
The eigenvalues and eigenvectors of the equilibria of the desingularized system~\eqref{E:FHN-desing} are obtained from direct calculation from the Jacobian of $J_{u\theta}$ (the matrix $J_{u\theta}$ was introduced in the proof of Proposition~\ref{Th:saddles-nodes}).
\end{proof}

The folded singularities on $\Fl$ are of particular interest in our analysis. One can estimate the relative order of magnitude of their eigenvalues and eigenvectors when $\delta$ becomes very small (e.g., $\delta \ll \eps$). However, this would be an extreme case of parameter choice that we do not actually consider in this paper (see Figure~\ref{fig:bifurcation_E_vs_w_region_I} and Figure~\ref{F:ROI-new}; we are mostly interested in the cases when $\delta = \bigO(\sqrt{\eps})$ or $\bigO(\eps)$). For comparison purposes, we include here a short discussion of the case $\delta \ll \eps$ then comment on the applicability of those results, under certain circumstances, to the case $\delta = \bigO(\eps)$. By expansion of the eigenvalues in Proposition~\ref{Th:eig}, we find that:
\[
\begin{cases}
\lambda_{S\ell,1} = -1 - \Lambda \, \delta + 2 \Lambda^2 \, \delta^2 + \bigO (\delta^3) \\
\lambda_{S\ell,2} = \phantom{-1 - }\, \,\, \Lambda \delta -  2 \Lambda^2 \, \delta^2 + \bigO (\delta^3)
\end{cases}
\quad \mbox{as} \quad \delta \to 0
\]
 and 
\[
\begin{cases}
\lambda_{N\ell,1} = -1 + \Lambda \, \delta + 2 \Lambda^2 \, \delta^2 + \bigO (\delta^3)  \\
\lambda_{N\ell,2} =  \phantom{-1} \, - \Lambda \delta -  2 \Lambda^2 \, \delta^2 + \bigO (\delta^3)
\end{cases}
\quad \mbox{as} \quad \delta \to 0
\]
where $\Lambda = 2(E^2 b^2 - \mu^2)$. Given the values of $b$ and $\mu$ in the FHN model ($b=0.8$, $\mu \approx 0.233$), if $E=0.6$ then $\Lambda \approx 0.35$ (see Figure~\ref{fig:bifurcation_E_vs_w_region_I} for the choice of $E$). For the expansions above to be valid we would need $\delta$ to be at order $\bigO(\eps \sqrt{\eps})$ or lower ($\eps=0.08$ implies $\eps \sqrt{\eps} \approx 0.0226$), though the approximations may still be relevant at $\delta=\eps$.

\subsection{Deriving the super-critical manifold}\label{Subsec: SuperSlowManifold}
If we fix $\varepsilon>0$ and instead take $\delta\rightarrow0$ in the slow-system~\eqref{E:FHN-uvth-slow}, the two-dimensional \textit{slow layer} dynamics are given by
\begin{equation}\label{E:FHN-uvth-slow-layer}
    \begin{aligned}
        \varepsilon\frac{du}{d\tau_1} &= -v + F(u) \\
        \frac{dv}{d\tau_1} &= u - bv + \mu - Eb \cos\left(\theta - \pi/2\right) \\
        \frac{d\theta}{d\tau_1} &= 0
    \end{aligned}
\end{equation}
where we have used that $R_\delta\rightarrow Eb$ and $\varphi_\delta\rightarrow\pi/2$ as $\delta\rightarrow0$.
The slow layer problem describes dynamics in planes with fixed $\theta = \theta_0$. 
The equilibria of \eqref{E:FHN-uvth-slow-layer} correspond to the manifold \begin{equation}\label{E:SuperSlowManifold}
    \mathcal{Z} = \left\{(u,v,\theta_0)\in\mathbb{R}^2\times [0,2\pi) \, |\, v = F(u), \  u-bv+\mu - Eb\cos(\theta_0-\pi/2) = 0\right\}.
\end{equation}
We call the one-dimensional curve $\mathcal{Z}$ the \textit{super-critical manifold} as in \cite{LetsonEtAl2017} (or alternatively, ``2-manifold" as in \cite{kaklamanos2022bifurcations}) to distinguish it from the critical manifold $\CM$. Since the first condition that defines $\mathcal{Z}$ is given by $v = F(u)$, the super-critical manifold is a curve along the critical manifold $\mathcal{M}$ (see magenta curves in Figure~\ref{fig:CM_3D_2D}). Note that we derived $\mathcal{Z}$ from the limit as $\delta\rightarrow0$ rather than taking first $\varepsilon\rightarrow0$ then $\delta\rightarrow 0$.

\begin{figure} 
    \centering
\includegraphics[width=1\textwidth]{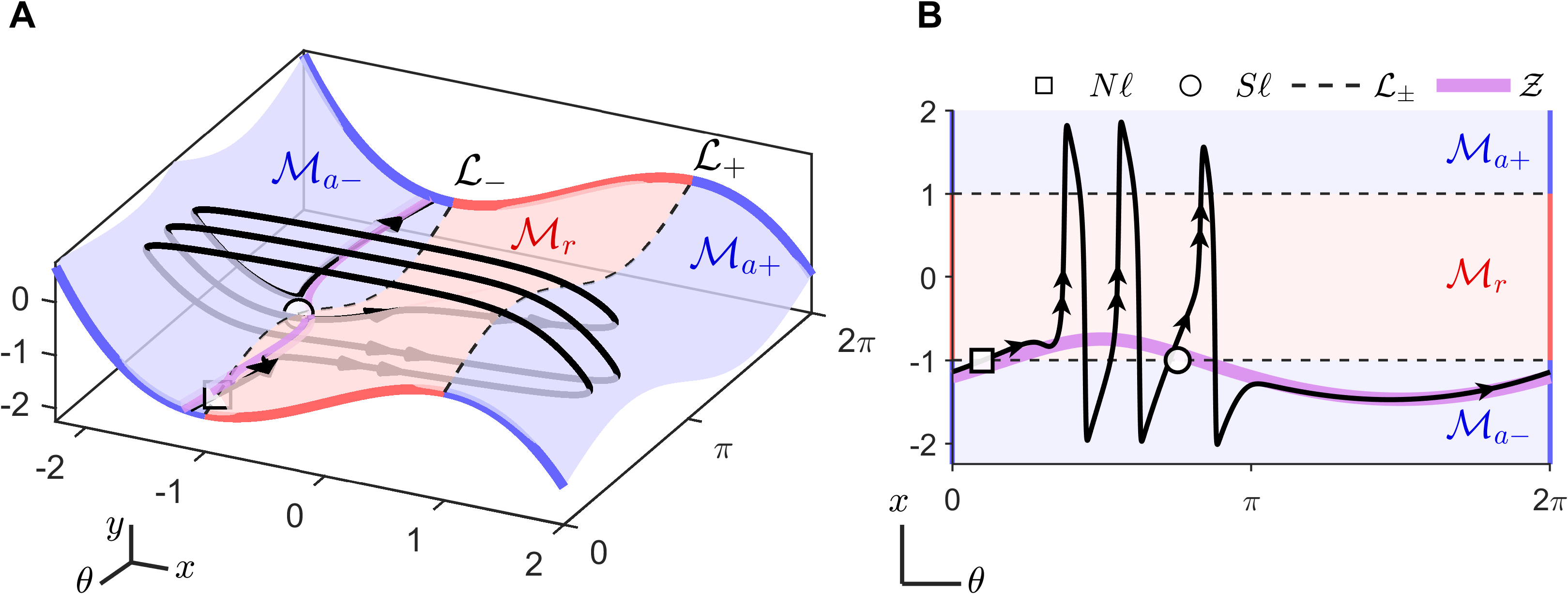}
    \caption{
    \textbf{Bursting solution with critical and supercritical manifold.} 
    \textbf{A} The critical manifold, $\CM$, is separated into attracting ($\CMl$, $\CMr$; blue) and repelling ($\CMc$; red) regions. 
    The attracting and repelling regions intersect along fold lines, $\mathcal{L}_{\pm}$. 
    The folded node ($\NL$) and \Rev{folded} saddle ($\SL$) are shown as a square and circle, respectively, along the fold lines.
    The supercritical manifold, $\mathcal{Z}$ (magenta), is a curve along the critical manifold.
    A bursting trajectory of system \eqref{E:FHN-1} with three spikes (forcing parameters $E=0.55$, $\omega=0.0149354$) is shown with $\theta =\omega t$ over one input period.
    This trajectory contains canard segments, following passage near the folded node and folded saddle, that remain near $\CMr$ for non-negligible time on the fast timescale.
    \textbf{B} The same trajectory projected onto the $\theta$--$x$ axis. Colors and symbols follow the same conventions as in panel \textbf{A}.
    }
    \label{fig:CM_3D_2D}
\end{figure}

Changing to the super-slow time, $\tau_2 = \delta \tau_1 = \delta \varepsilon t$, and letting $\delta\to0$, system \eqref{E:FHN-uvth-slow} becomes 
\begin{subequations}\label{E:FHN-SS-Slow}
    \begin{align}
        0 &= -v + F(u) \\
        0 &= u - bv + \mu - Eb \cos\left(\theta - \pi/2\right)  \\
        \frac{d\theta}{d\tau_2} & = 1
    \end{align}
\end{subequations}
which governs the super-slow reduced problem along $\mathcal{Z}$. 
At hyperbolic points of $\mathcal{Z}$, Fenichel Theory holds and the super-critical manifold $\mathcal{Z}$ perturbs to the super-slow manifold $\mathcal{Z}^{\delta}$. 
The super-slow manifold inherits the stability properties from $\mathcal{Z}$ with solutions flowing towards attracting and away from repelling regions of $\mathcal{Z}^{\delta}$. 
Just as in the slow-time analysis, GSPT breaks down at non-hyperbolic points of $\mathcal{Z}$. 
To find the points where the hyperbolic property is lost, we compute the Jacobian of slow layer (first two equations) in system \eqref{E:FHN-uvth-slow-layer},
\begin{equation}
    J_{SL} = 
        \begin{bmatrix}
            \varepsilon^{-1}F'(u) & -\varepsilon^{-1} \\
            1 & -b
        \end{bmatrix}
\end{equation}
Since the parameter $b$ is between zero and one, the determinant of $J_{SL}$ is always positive, $\det\left(J_{SL}\right) = \varepsilon^{-1}\left(1-bF'(u)\right)=\varepsilon^{-1}\left(1-b(2u-u^2)\right)>0$. Therefore, hyperbolicty is lost at Hopf bifurcation points along $\mathcal{Z}$ where $\tr(J_{SL})=0$, which occurs when $u_{dh,\pm} = 1\pm\sqrt{1-\varepsilon b}$. 

Converting to the original $(x,y,\theta)$ coordinates (with $x=u-1$, see Sec.~\ref{S:derive-desingSyst}), a Hopf bifurcation occurs at $x_{dh,\pm} = \pm\sqrt{1-\varepsilon b}$ along $\mathcal{Z}$. This Hopf bifurcation does not belong to the original system \eqref{E:FHN-2}, but instead to the {slow layer} problem \eqref{E:FHN-uvth-slow-layer}. Such bifurcations are called \textit{delayed Hopf bifurcations}
\Rev{\cite{neishtadt1987persistence, neishtadt1988persistence, hayes2016geometric, LetsonEtAl2017}.}

\begin{remark}
The critical and super-critical manifolds $\mathcal{M}$ and $\mathcal{Z}$ were derived considering the limits $\varepsilon\to0$ and $\delta\to0$ independently. That is, in each case one timescale parameter was held fixed while the other was used in the singular limit. A more interesting question would be to ask how the dynamics of the FHN model changes when both parameters $\eps$ and $\delta$  are small and of similar magnitude, where $\eps$ is associated with the intrinsic timescale of the neuronal dynamics, while $\delta$ controls the frequency of the external oscillatory driving force. This is the {\it three-timescales dynamical systems problem} that we investigate next. 
\end{remark}

\section{Spike-adding mechanisms in the FHN model with periodic forcing}\label{Sec:Spike_adding}
In this section, we will discuss  possible spike adding mechanisms for different values of the input's parameters.    

\subsection{Region of interest for bursting behavior in the 
$(\Rev{\Rev{\delta}},E)$-plane}
First, we compare our theoretical results (summarized in Figure~\ref{F:ROI}) with our numerical observations of the FHN model with periodic input (shown in Figure~\ref{fig:bifurcation_E_vs_w_region_I}). We find that parametric regions where bursting solutions maintain a given number of spikes, and the boundaries between them,  appear primarily between the curves of folded saddle-nodes on $\Fl$ and $\Fr$ respectively; see Figure~\ref{F:ROI-new}. This area is the union of regions II and III where the model has two and only two folded singularities (\Rev{folded} saddle plus \Rev{folded} node/focus), with both of them located on $\Fl$.

An exemplar bursting trajectory of the FHN model \eqref{E:FHN-1},  projected on the  $(x,y,\theta)$ space,  is illustrated in Figure~\ref{fig:CM_3D_2D}. The figure was generated by simulating system~\eqref{E:FHN-1} for a pair of parameters $(\Rev{\delta}, E)$ chosen in region II. Note that $x$, $y$ in \eqref{E:FHN-1} are calculated from $u$, $v$ and $\theta$ in  
 \eqref{E:FHN-3} as described in Sec.~\ref{S:derive-desingSyst}.  
 
As stated in Sec.~\ref{Sec:Full_Model}, our initial plan was to analyze system \eqref{E:FHN-1}'s bursting behavior for input parameters $(\omega,E)\in(0,\varepsilon]\times[0,1]$. However, our theoretical results point to a more ``natural" set of constraints for a parameter  region of interest: the area between the curves delineated by  $E^{*}_{\ell,\delta}$ and $E^{*}_{r,\delta}$ in Figures~\ref{F:ROI} and \ref{F:ROI-new} (i.e., folded SN on $\Fl$ and folded SN on $\Fr$). Note that 
 while $E^{*}_{r,1}$ is less than one,  it is nevertheless very close to it ($\delta=1$ implies $\omega=\eps$). Most importantly, these constraints were derived from and are closely related to the dynamics of the unforced FHN model, as they depend on its intrinsic parameters $a$ and $b$; see the inequalities in \eqref{E:condE}.

To conclude, we will work in the remaining sections of the paper (Sec.~\ref{S:spike-mech} and beyond) under the assumption that $E^{*}_{\ell,\delta} < E < E^{*}_{r,\delta}$ - see Figure~\ref{F:ROI-new}, \Rev{the region between the two curves labeled ``folded SN on $\Fl$" and ``folded SN on $\Fr$"}.

\begin{figure}
    \centering  \includegraphics[scale=1]{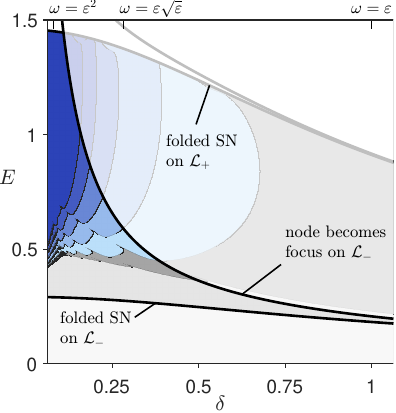}
    \caption{\textbf{Parameter region of interest for bursting}. 
      Cusp-like boundaries for changing spike counts are located where a folded node and folded saddle exist on $\Fl$ (region II in Figure~\ref{F:ROI}; see also Figure~\ref{fig:bifurcation_E_vs_w_region_I}). The regions below the saddle-node (SN) bifurcation on $\Fl$ ($E=E^{*}_{\ell,\delta}$) and above the curve  marking the transition from folded node to folded focus ($E=E^{**}_{\ell,\delta}$) are grayed-out. The parameter region above the SN on $\Fr$ ($E=E^{*}_{r,\delta}$) is shown in white. In our analysis we constrained input forcing parameters to $E^{*}_{\ell,\delta} < E < E^{*}_{r,\delta}$. 
      \Rev{Note that the horizontal axis closely resembles that of Figure~\ref{fig:bifurcation_E_vs_w_region_I} where $\omega=\varepsilon\delta\in[0,0.08]$ and $\varepsilon=0.08$.} }
  \label{F:ROI-new}
\end{figure}

\begin{figure} 
	\begin{center}
		\includegraphics[width=\textwidth]{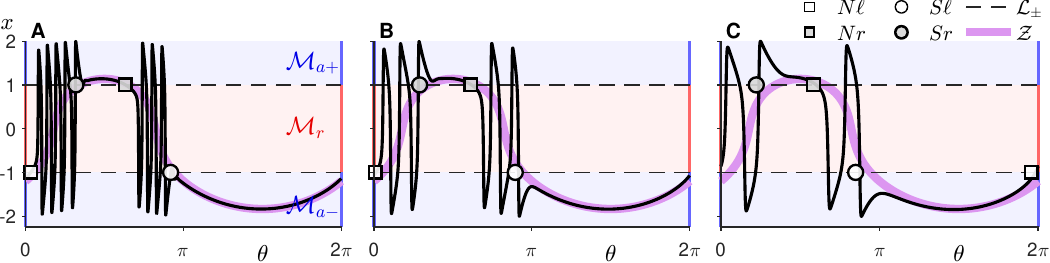}
        \caption{\textbf{Held bursting dynamics.} For input parameters satisfying $E> E_{r,\delta}^*$ (orange-shaded areas in Figure~\ref{F:ROI}), two folded equilibria emerge  on $\Fr$ (shown with gray markers between $\CM_{r}$ and $\CM_{a+}$). Notations: Branches of the critical manifold $\CM_{a\pm}$ in blue, $\CM_r$ in red; Folded curves $\Fl$, $\Fr$, dashed lines; Super-critical manifold $Z$, in purple; Folded saddles as circles, and folded nodes/foci as squares.
 Panels \emph{\textbf{A}}, \emph{\textbf{B}}, \emph{\textbf{C}} illustrate solutions to \eqref{E:FHN-1} for input parameters $(\omega,E)$  chosen from Regions IV, V, and VI in Figure~\ref{F:ROI}, respectively. The solutions are projected onto the $\theta$\,--\,$x$ plane. Note that in all three cases, there is also a pair of folded equilibria on $\Fl$ (shown with white markers between $\CM_{a-}$ and $\CM_{r}$). In these three parameter regimes, a trajectory can be ``held'' to the slow dynamics along $\CMr$ within a burst's spiking phase. We do not study these types of dynamics in our paper. }\label{fig:held-bursting}
    \end{center}
\end{figure}

\subsection{Complementary FHN dynamics in other input-based regions of the parameter plane}
Before analyzing bursting patterns in the FHN model with periodic forcing, we review again the results obtained in Figure~\ref{F:ROI}.  The figure illustrates six distinct regions of the $(\Rev{\delta}, E)$ parameter plane, characterized by the stability types of folded equilibria and their existence  when crossing $\Fl$ or $\Fr$.

\Rev{As also shown in Figures~\ref{fig:bifurcation_E_vs_w} and~\ref{fig:bifurcation_E_vs_w_region_I} numerically, Region I corresponds to parameter values for which the system does not admit bursting solutions. Therefore, we do not consider this region in our study.}

In Region II the system admits only one folded saddle and one folded node on $\Fl$, while in Region V  it admits a folded saddle and a folded focus on $\Fl$ as well as a folded saddle and a folded node on  $\Fr$. On the other hand, in Region IV  the system has exactly two folded saddles and two folded nodes, one of each on $\Fl$ and $\Fr$.

Lastly,  Region III contains the parameters for which the system admits only one folded saddle and one folded focus on $\Fl$, while in  Region VI  the system admits exactly two folded saddles and two folded foci, one of each on $\Fl$ and $\Fr$.

Note that Regions IV, V, and VI of Figure~\ref{F:ROI} (in orange) typically admit solutions whose spiking phase is briefly interrupted within a single input period. 
The spike-interruption arises with the emergence of the folded equilibria on the fold curve $\Fr$. 
Within the spiking phase, the trajectory is temporarily ``held" to slow dynamics along $\CMr^{\varepsilon}$, inducing brief periods of saturation \cite{davison2019mixed} or small amplitude oscillations in the vicinity of $\CMr^{\varepsilon}$; see Figure~\ref{fig:held-bursting} for exemplary solutions. Such oscillations have been called in the past ``held-small amplitude oscillations'' \cite{espanol2015BZ}. In this paper we are interested in characterizing the spike addition process in {\it bursts with uninterrupted spiking phases}. Therefore, we will limit our analysis to the spike-adding mechanisms identified exclusively  in Regions II and III. 

\subsection{The role of folded saddle canards and folded node canards in bursts spike adding} \label{S:spike-mech}
The number of spikes in a burst observed in either Region II or III is influenced by the proportion of the input period that occurs after a trajectory has passed through the folded node/focus on $\Fl$, but before it reaches the folded saddle. 
The phase-distance between the folded equilibria, however, is not the sole factor for the spike count within a burst. 

Firstly, a bursting sequence does not necessarily end exactly at the folded saddle.
Instead, spike termination is influenced by a trajectory's location along $\CM^{\varepsilon}$ as it passes near the plane $\theta=\theta_{S\ell}$.

Secondly, in the case of a folded node (Region II of Figure~\ref{F:ROI}), trajectories experience a delay to spiking after passing near the folded node during an excursion on the repelling portion of $\CM_\varepsilon$.
The delay to spiking shortens the phase-distance from a spike-initiation point to the folded saddle, reducing the time spent in the active phase. In the case of a folded focus (Region III of Figure~\ref{F:ROI}), solutions are guided towards the fold $\Fl$ by the equilibrium's local spirals, where they make a fast jump without an excursion along $\CMc$.

Furthermore, spike initiation out of the folded node is influenced both by canard and delayed-Hopf mechanisms 
\cite{kaklamanos2022bifurcations,LetsonEtAl2017} and the influence of these separate mechanisms shifts as the forcing frequency changes.
As $\omega$ transitions from $\mathcal{O}(\varepsilon)$ to $\mathcal{O}(\varepsilon\sqrt{\varepsilon})$, the folded node canard dynamics dominate, leading to spike initiation  at a jump-across canard or after a jump-back canard followed by a fast jump across upon returning to the fold line.
Then, as $\omega$ transitions to $\mathcal{O}(\varepsilon^2)$, the underlying delayed-Hopf bifurcation, characterized by a small-amplitude ringing oscillations, manifests at spike initiation.
The interaction between the delayed spike-initiation mechanism and the trajectory's proximity to the folded saddle results in the observed non-monotonic relationships between the number of spikes and the forcing parameters.

A detailed analysis of the delayed-Hopf mechanism for small $\omega$ and its interaction with canard mechanism is not in the scope of this paper, but will be considered in future work.

In what follows, we  focus on the parameter regions II and III of Figure~\ref{F:ROI} and 
discuss two main spike-adding mechanisms:
\begin{itemize}
\item  through the folded saddle canard (in both Regions II and III; Sec.~\ref{subsec:adding-spikes-at-saddle}), and
\item through the folded node canard (in Region II; Sec.~\ref{subsec:adding-spikes-at-node}). 
\end{itemize}

\subsection{Adding spikes through the folded saddle canard}\label{subsec:adding-spikes-at-saddle}
In this section, we describe folded-saddle canard dynamics which occur in Regions II and III of Figure~\ref{F:ROI}.

By Proposition~\ref{Th:FoldedEq}, the desingularized system~\eqref{E:FHN-desing} admits a folded saddle on $\Fl$ as long as 
$ E>E^*_{\ell,\delta}$.
However, the desingularized system was obtained by a local state-dependent time rescaling that reverses the direction of the flow on $\CMc$ compared to the flow of \eqref{E:FHN-slow-reduced} and \eqref{E:FHN-1}.
As such, desingularized solutions beginning in $\CMc$ that flow towards the \Rev{folded} saddle point $P_{S\ell}$ of coordinates $(u_{S\ell}, v_{S\ell}, \theta_{S\ell})$ (see Proposition~\ref{Th:FoldedEq})  are reversed in the slow-reduced and in the full problem. In particular, the points that lie on $W_{S\ell}^{s} \cap \CM_r$, where $W_{S\ell}^{s}$ is the stable manifold of  the \Rev{folded} saddle, are repelled in the full system \eqref{E:FHN-1} away from $P_{S\ell}$.
The curve $W_{S\ell}^{s}$ provides a path for slow solutions to pass through the \Rev{folded} saddle from $\CMl$ into $\CMc$.
This curve forms a \textit{vrai canard} (or \textit{true canard}), which perturbs to a canard segment in $\CMl^{\varepsilon}\cap\CM_{r}^{\varepsilon}$ for $\varepsilon>0$ \cite{szmolyan2001R3}. 

For the desingularized system~\eqref{E:FHN-desing} the vrai canard acts as a separatrix, distinguishing solutions that pass between $\CMl$ and $\CMc$ from those that do not. 
For the full system~\eqref{E:FHN-1} the canard serves as a threshold, distinguishing solutions that remain in the neighborhood of $P_{S\ell}$ while crossing into $\CM_{r}^{\varepsilon}$ from those that jump across to reach \Rev{$\CMr^{\varepsilon}$}   immediately \cite{wieczorek2011compost, wechselberger2013nonauto}.

 \begin{figure}
	\begin{center}
		\includegraphics[width=\textwidth]{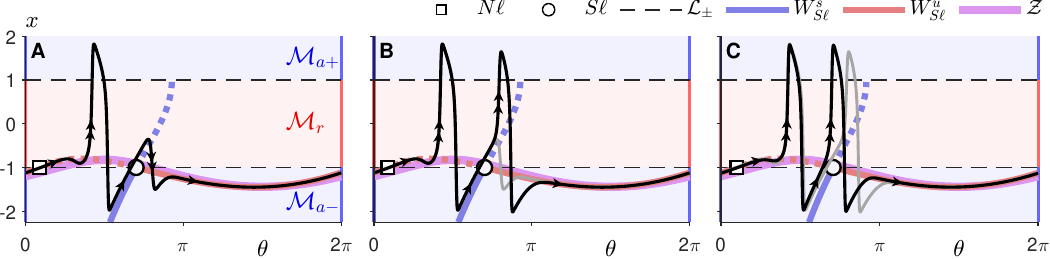}
		\caption{\textbf{Spike adding at folded saddle.} Trajectories from the full three-dimensional system \eqref{E:FHN-3} projected onto the $\theta$\,--\,$x$ plane for $E=0.482$ and \emph{(\textbf{A})} $\omega = 0.02206875$,  \emph{(\textbf{B})} $\omega=0.0220625$, and \emph{(\textbf{C})} $\omega=0.02$ (all in Region II of Figure~\ref{F:ROI}).   As a trajectory nears the termination of a burst sequence, its intersection with $\CMl$ approaches $W^s_{S\ell}$. The proximity to $W^s_{S\ell}$ determines if there is a jump-back canard (\emph{\textbf{A}}, no spike), a jump across canard (\emph{\textbf{B}}, spike), or a jump at the fold curve  (\emph{\textbf{C}}, spike) followed by a subsequent return to the super-slow manifold $Z$ (in purple). The transition from jump-back to jump-across canard marks the addition of a spike near the folded saddle (indicated by a circle). In panel \emph{\textbf{B}} (resp. \emph{\textbf{C}}) the trajectory from \emph{\textbf{A}} (resp. \emph{\textbf{B}}) is shown in gray for comparison. In the desingularized system, flows are reversed along $\CMc$; 
      the stable manifold of the \Rev{folded} saddle $W^{s}_{S\ell}$  (in blue) and the unstable manifold of the \Rev{folded} saddle $W^{u}_{S\ell}$ (in red; very close to $Z$) are drawn as dashed lines in this region.}\label{fig:spike-adding-saddle}
	\end{center}
\end{figure}
In $(x,y,\theta)$ coordinates, the curve $W^{s}_{S\ell}$ can be mapped onto the critical manifold $\CM$ and traced until it intersects  $x=-2$, which is the estimated lower bound for $x$-values in spiking solutions (Figure~\ref{fig:spike-adding-saddle}; also, see comment at the end of Sec.~\ref{Subsec: Criticalmanifold}).  
We denote $\Sigma_{\SL,-2}$ as the plane containing this intersection point, with $\theta$ fixed.
Similarly, a bursting solution approaches the  lower bound plane $x=-2$ multiple times, inducing an increasing sequence $\{\theta^1_{-2}, \dots,\theta^m_{-2}\}$. This sequence marks the $\theta$-values when the solution reaches its local minimum in the $x$-variable. 
The proximity of the sequence $\{\theta^{1}_{-2},\dots,\theta^m_{-2}\}$ to $W^{s}_{\SL}$ provides an estimate for the sequence's proximity to the separatrix, which determines whether a spike is added at the folded saddle or not.

First, consider the case when the $\theta$-sequence contains a member that lies close to $\Sigma_{\SL,-2}$ with a solution that has crossed the separatrix marked by the perturbation in $\CMl^{\varepsilon}\cup\CMc^{\varepsilon}$ of the vrai canard.
In this case, $\theta^{m}_{-2}$ (belonging to the last spike) is slightly beyond (to the right of) where the separatrix intersects $x=-2$. Note that in Figure~\ref{fig:spike-adding-saddle}, we plot in dark blue the folded saddle's stable manifold, $W^{s}_{\SL}$, which is a visual surrogate for a neighborhood of the separatrix for $\varepsilon>0$ (the $\varepsilon$-perturbed separatrix itself is not shown).
Solutions will flow towards the folded saddle into $\CMc^{\varepsilon}$ and jump back towards $\CMl^{\varepsilon}$. 
This forms a `canard without-head' segment, which does not form a spike; see Figure~\ref{fig:spike-adding-saddle}A.
The trajectory  returns to the super-slow manifold $\SCM^{\delta}$ until the next bursting sequence.

If instead $\theta^{m-1}_{-2}$ (following the second-to-last spike) approaches the intersection of the separatrix  with $\Sigma_{\SL,-2}$, and remains to the left of the separatrix, the jump-back canard gives rise to a jump-across canard. 
Here, solutions will approximately follow $W^{s}_{S\ell}$ past the fold line and into $\CMc^{\varepsilon}$, where it remains for $\mathcal{O}(1/\varepsilon)$ on the fast timescale, forming a canard segment, before exiting $\CMc^{\varepsilon}$ as a fast jump across to $\CMr^{\varepsilon}$; see Figure~\ref{fig:spike-adding-saddle}B. 
The solution follows the slow dynamics along $\CMr^{\varepsilon}$, and makes a fast jump at $\Fr$ back towards $\CMl^{\varepsilon}$. 
Upon reaching $\CMl^{\varepsilon}$, the solution flows towards and then follows the super-slow manifold $\SCM^{\delta}$, where it remains until approaching the folded node for the next bursting cycle.
We consider this case as {\it a spike added at the folded-saddle}, which is characterized by a so-called `canard with-head' segment in the final spike of the bursting oscillation.

If no member of the $\theta$-sequence is close  to the intersection of the separatrix with $\Sigma_{\SL,-2}$, then we do not consider spike-adding to occur at the \Rev{folded} saddle.
In this case, upon reaching $\theta^{m-1}_{-2}$, a solution flows towards the fold line $\Fl$  
 and makes a fast jump (relaxation oscillation) at it.
The return dynamics to $\CMl^{\varepsilon}$ bring the solution well past separatrix, leading it back to $\SCM^{\delta}$, as in the previous cases (Figure~\ref{fig:spike-adding-saddle}C).

\subsection{Adding spikes through the folded node canard}
\label{subsec:adding-spikes-at-node}

In this section, we describe folded-node canard dynamics which occur exclusively in Region II of Figure~\ref{F:ROI}. 

At the onset of an input period, bursting solutions are funneled towards the repelling region of the slow manifold through the folded node $P_{N\ell}$ of coordinates $(u_{N\ell}, v_{N\ell}, \theta_{N\ell})$ (see Proposition~\ref{Th:FoldedEq}). Upon exiting the \Rev{folded} node, trajectories remain near the repelling manifold, forming canard segments at the start of a burst. The segment escapes the repelling region by either jumping back to the lower attracting branch where it originated (Figure~\ref{fig:spike-adding-node}A,\,B; see the trajectory immediately past the square mark), or by jumping across to the upper attracting branch (Figure~\ref{fig:spike-adding-node}C):

$-$ For the jump-back case, the burst begins with a  small amplitude oscillation formed by the folded node canard segment on the slow timescale. Next, the jump-back trajectory possibly induces:  
i) a relaxation-oscillation (simple jump at the fold curve), 
ii) a canard segment with spike, 
iii) a canard segment with no spike, or 
iv) a return to the super-slow manifold.  The outcome depends on the interaction of the trajectory with the folded saddle as described in Sec.~\ref{subsec:adding-spikes-at-saddle}. 

$-$ The jump-across following passage through the folded node contributes a spike, which contains a canard segment, to the burst. The trajectory returns to the lower attracting branch, according to the return mechanism described in Sec.~\ref{subsec:adding-spikes-at-saddle}. 

 A jump-back out of the folded node does not necessarily decrease the \textit{total} number of spikes possible in a bursting oscillation. 
Suppose, for example, that an initial jump-across occurs and $n-1$ spikes follow (giving a total of $n$ spikes) before spike-termination at the \Rev{folded} saddle. A fast jump-back and subsequent return to the fold curve may leave enough time (or phase $\theta(t)$) for $n$ spikes before termination at the \Rev{folded} saddle.
To alter the number of spikes in a burst, the initial jump-back  must significantly change the trajectory's ultimate proximity to $W_{S\ell}^s$, the  stable manifold of the \Rev{folded} saddle, to undergo the spike-changes  described in the previous section.

To summarize, at the start of the input period, solutions start along the super-slow manifold, $\mathcal{Z}^{\delta}$, and are funneled towards the folded node.
After exiting the folded node, the solution stays along $\CM_r^\eps$, forming a folded node canard segment. 
This segment gives rise to a fast jump either back to $\CMl^{\varepsilon}$ or up to $\CMr^{\varepsilon}$. 
If the solution has not yet reached a neighborhood of the folded saddle, relaxation oscillation segment(s) ensue, in which the solution approaches the fold $\mathcal{L}_{\mp}$ along $\CM^{\varepsilon}_{a\mp}$ and makes fast jumps to $\mathcal{L}_{\pm}$.
If a neighborhood of the folded saddle is reached, a folded saddle canard segment is formed, giving one final jump either back to $\Fl$ or across to $\Fr$.
Finally, the solution returns to $\CMl^{\varepsilon}$, flows back towards $\mathcal{Z}^{\delta}$, and follows this curve until the bursting sequence is re-initiated.

 \begin{figure} 
	\begin{center}
		\includegraphics[width=\textwidth]{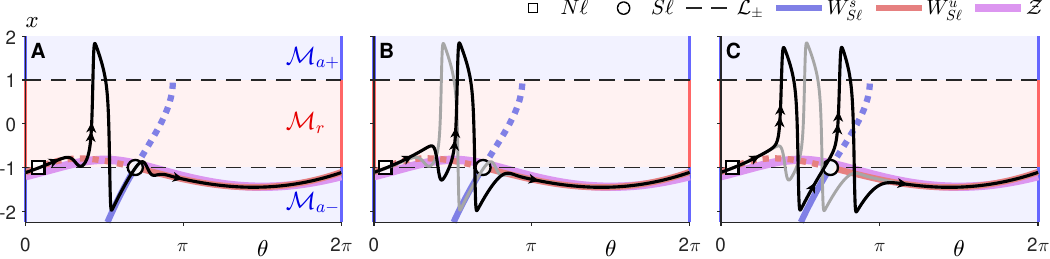}
		\caption{\textbf{Spike adding at folded node.} Trajectories from the full three-dimensional system \eqref{E:FHN-3} projected onto the $\theta$\,--\,$x$ plane for $E=0.482$ and \emph{(\textbf{A})} $\omega = 0.0236$,  \emph{(\textbf{B})} $\omega=0.02506875$, and \emph{(\textbf{C})} $\omega=0.025075$ (all in Region II of Figure~\ref{F:ROI}). Trajectories are funneled towards the folded node (square) and into $\CMc$. After an excursion along $\CMc$, trajectories either jump back towards $\CMl$ (\emph{\textbf{A}}, \emph{\textbf{B}}, no spike) or jump across towards $\CMr$ (\emph{\textbf{C}}, spike). From \emph{\textbf{A}} to \emph{\textbf{B}}, the escape from $\CMc$ towards $\CMl$ yields a small amplitude oscillation with increasing magnitude. From \emph{\textbf{B}} to \emph{\textbf{C}}, the small amplitude oscillation gives rise to a jump-across towards $\CMr$. The jump-across marks the addition of a spike out of the folded node canard. In panel \emph{\textbf{B}} (resp. \emph{\textbf{C}}) the trajectory from \emph{\textbf{A}} (resp. \emph{\textbf{B}}) is shown in gray for comparison. In the desingularized system, flows are reversed along $\CMc$, so  the stable manifold of the \Rev{folded} saddle (circle) which is $W^{s}_{S\ell}$ (in blue), and the unstable manifold $W^{u}_{S\ell}$ of the \Rev{folded} saddle (in red; very close to the super-critical manifold $Z$)
   are drawn as dashed lines in this region.}\label{fig:spike-adding-node}
	\end{center}
\end{figure}

\begin{remark}
Folded node canard dynamics occur exclusively in Region II of Figure~\ref{F:ROI}. 
As the parameter regime transitions from Region II into III, the folded node becomes a folded focus and the canards exist only at the folded saddle. The \emph{cusps} in the spike-count boundaries  appear only in Region II as well (see Figure~\ref{F:ROI-new}) $-$  they arise from the complex interaction of the jump-back dynamics at the folded node with the jump-across dynamics at the folded saddle. In Region III where there is a folded focus, the spike-adding is influenced by the decay rate of the stable spirals near the \Rev{folded} focus and the \Rev{folded} saddle-canard dynamics. However, these interactions do not induce cusp-like boundaries (transparent-blue area in Figure~\ref{F:ROI-new}).
In the next section, we  examine the interaction between the  folded node and folded saddle dynamics in more detail.
\end{remark}

\section{Determining spike count boundary}\label{Sec:spike-count-boundary}
The mechanisms discussed in Sec.~\ref{Sec:Spike_adding} explain how spikes are added or removed by linking trajectories to folded equilibria. 
Additionally, we are interested in how these mechanisms impact the boundaries in the parameter space that separate regions with different spike counts (Figure~\ref{fig:bifurcation_E_vs_w_region_I}). 
Although we do not provide an analytical derivation of these boundaries, further numerical computations allow us to describe how their features relate to the spike-adding mechanisms.

\subsection{Level sets of trajectory $L^2$ norm}\label{Subsec:L2-norm}
 Previous studies of numerically computed  canards identified the $L^2$ norm as a suitable metric to quantify canard solutions \cite{ brons2006MMO, desroches2010numerical, wechselberger2005bifurcation}. In particular, the $L^2$  norm was used to distinguish solutions at bifurcations of canards. 
 We define here the $L^2$ norm as
\begin{equation}\label{E:L2-norm}
\left(\frac{1}{T}\int_0^\top \left[x(t)^2 + y(t)^2\right]\,\mathrm{d}t\right)^{1/2}
\end{equation}
where $x(t)$ and $y(t)$ are the solutions of the FHN system \eqref{E:FHN-1}  estimated numerically, and the norm is normalized to the input period $T = 2 \pi / \omega$.  
We compute the $L^2$ norm for all numerical solutions in the $(\omega,E)$ parameter space of interest and visualize the level sets of the computed norms in Figure~\ref{fig:L2-bifurcation}. For comparison reasons, we superimposed in Figure~\ref{fig:L2-bifurcation} the spike-count boundaries from Figure~\ref{fig:bifurcation_E_vs_w_region_I} and showed a zoomed in area near a cusp.

\begin{figure} 
	\begin{center}
	\includegraphics[width=\textwidth]{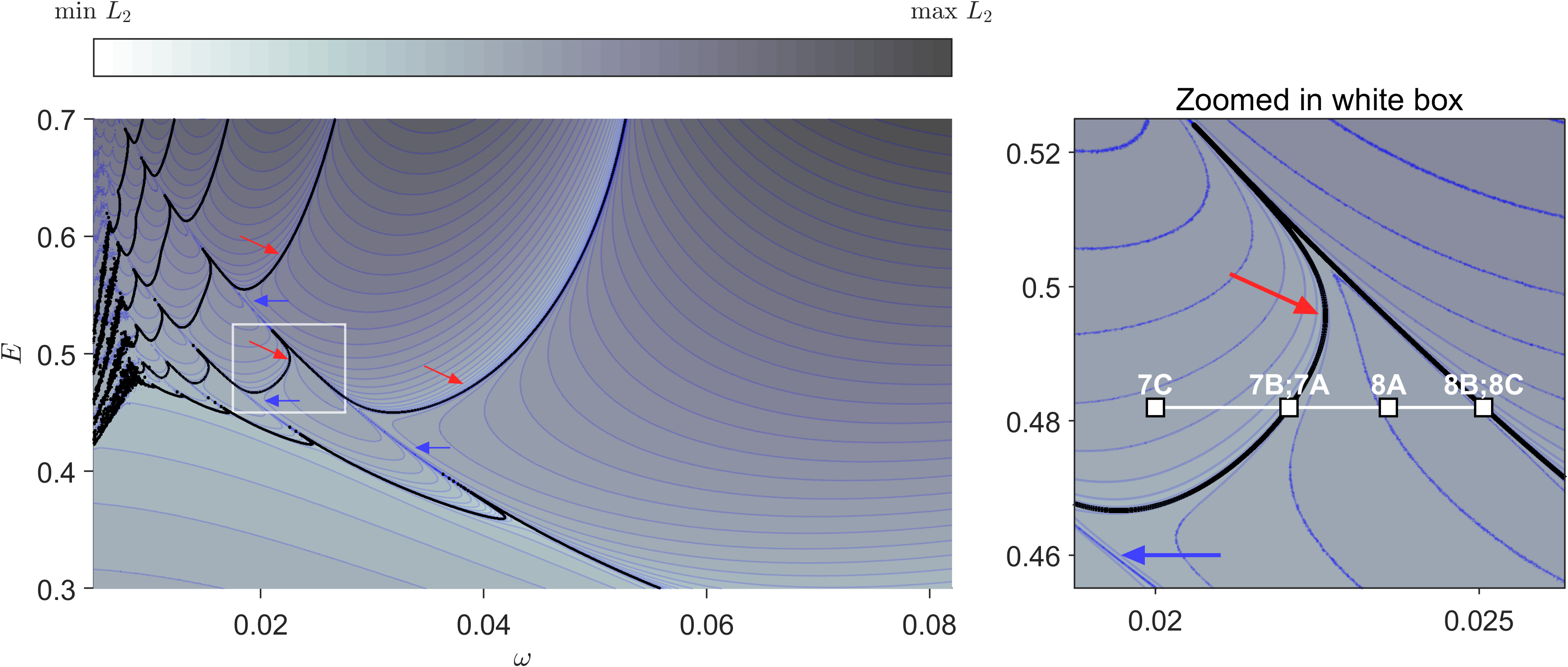}
		\caption{\textbf{Levels sets for the $L^2$ norm in simulated solutions.} The level sets were computed according to Eq.~\eqref{E:L2-norm} and super-imposed with the spike-count boundaries from Figure~\ref{fig:bifurcation_E_vs_w_region_I}. 
        Contours separating $L^2$ level sets are shown in transparent blue.
        The $L^2$ norm asymptotically approaches the super-imposed boundaries. 
        Horizontal blue arrows indicate parameter regimes near transitions from jump-back to jump-across (from left-to-right) canards at the folded node. 
        Slanted red arrows indicate regimes near transitions from jump-back to jump-across (from right-to-left) canards at the folded saddle.
        The zoomed in right-panel also marks the parameter locations used to generate the trajectories in Figures~\ref{fig:spike-adding-saddle} and~\ref{fig:spike-adding-node}. 
        Since the parameters in Figures~\ref{fig:spike-adding-saddle}A-B (resp. Figures~\ref{fig:spike-adding-node}B-C) are hardly distinguishable at the displayed scale, they are represented with a single box at the midpoint of the line connecting the parameter pairs.}\label{fig:L2-bifurcation}
	\end{center}
\end{figure}

The first observation is that the $L^2$ level sets funnel asymptotically along the spike-count boundaries (Figure~\ref{fig:L2-bifurcation}; red arrows).
This is unsurprising as one would expect a change in the number of spikes to induce a corresponding change in $L^2$.
Along these asymptotic boundaries, but away from the negatively-sloping component of the cusp, the number of spikes change due to the spike-adding mechanism at the folded saddle canard.
That is, as the parameter regime crosses these boundaries from right-to-left, a spike is added at the folded saddle, as in the transition shown in Figure~\ref{fig:spike-adding-saddle}A to Figure~\ref{fig:spike-adding-saddle}B.

We also observe level sets that  connect the cusps of the spike-count boundaries (Figure~\ref{fig:L2-bifurcation}; blue arrows).
Along these lines, the solution comes out of the folded node transition from a jump-back-canard to a jump-across canard (from left-to-right in parameter space), as shown in the transition from Figure~\ref{fig:spike-adding-node}B to Figure~\ref{fig:spike-adding-node}C. 
When these cusp-connecting-lines are away from spike-count boundaries (e.g., bottom right blue arrow in the left panel of Figure~\ref{fig:L2-bifurcation}), the transition from jump-back to jump-across at the folded node still occurs, but the spike-count does not change due to the trajectory's ultimate proximity to the folded saddle.

These two organizing features $-$ the asymptotic curves and the cusp-connecting level sets $-$ depict how the cusps form in the spike-count boundaries.
Fundamentally, cusps occur in parameter regimes in which the bifurcation between the jump-back and jump-across canards at the \Rev{folded} node leads to trajectories that ultimately land close enough to the stable manifold of the \Rev{folded} saddle to be influenced by saddle-canard dynamics.
Put another way, the cusps arise as the bifurcation curve for spike-adding at the \Rev{folded} node (Figure~\ref{fig:L2-bifurcation}; see the curves depicted by the blue arrows) collides with the bifurcation curve for spike-adding at the \Rev{folded} saddle (Figure~\ref{fig:L2-bifurcation};  see the curves depicted by red arrows).
Thus, we attribute the cusp formation to the interaction of the folded node and folded saddle dynamics.

\subsection{Phase distance}
The mechanisms of spike-adding at the folded node and folded saddle indicate how the spike-count boundaries might be derived analytically.
Estimating the time spent (or `phase' spent when considering the $\theta$\,--\,$x$ plane) on $\CMl^{\varepsilon}$ before the first jump on the fast timescale gives an approximation for the delay-to-spike within a burst.
Similarly, estimating where the stable manifold of the \Rev{folded} saddle intersects the line $x=-2$  allows an approximation for spike termination (recall that $\Sigma_{\SL,-2}$  denoted the plane containing this point).
In particular, one would consider the phase difference between these two estimates, relative to the intra-burst spiking frequency, to estimate the number of spikes in a burst.

\begin{figure}
    \begin{center}
    \includegraphics[scale=1]{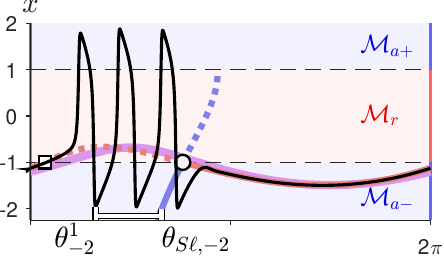}
    \caption{\textbf{Estimating duration of spiking phase.} A sample trajectory's return to the line $x=-2$ following its first spike is marked as $\theta^{1}_{-2}$. 
The intersection of the \Rev{folded} saddle's stable manifold with $x=-2$ is marked as $\theta_{\SL,-2}$.
The difference between these values, $\Delta \theta$, was used to estimate the number of spikes in a burst; see Eq.~\eqref{E:numSpikesUnshifted}. Circle/square markers and the color code for different curves are the same as in previous figures. }\label{fig:theta_diff_traj}
    \end{center}
\end{figure}

This approach is hindered by a few factors. First, predicting the bifurcation point between a jump-back and a jump-across canard at the \Rev{folded} node is difficult, as these dynamics likely involve bifurcations of canard solutions \cite{brons2006MMO, wechselberger2005bifurcation}.
Second, the stable manifold of the \Rev{folded} saddle comes from the singular limit as $\varepsilon\to0$, making it hard to exactly predict the separatrix that  distinguishes jump-back from jump-across solutions at the \Rev{folded} saddle.
And third, we do not establish an analytic relationship between the intra-burst spiking frequency and input parameters.

Despite these challenges, we will estimate proxies for each of these quantities.

For a simulated solution, we track it until first reaching $x=-2$ following a spike; this corresponds to $\theta^1_{-2}$ from the $\theta$-sequences described in Sec.~\ref{subsec:adding-spikes-at-saddle}.
We then calculate a fifth-order expansion for $W^{s}_{\SL}$ (see Appendix \ref{Sec:appendix}), which we use as an estimate for the separatrix at the folded saddle. We find where this curve intersects $x=-2$ and denote the corresponding phase value as $\theta_{\SL,-2}$ (Figure~\ref{fig:theta_diff_traj}).
Finally, we estimate the intra-burst frequency from the frequency of the constantly forced FHN neuron (i.e., when the forcing term in system \eqref{E:FHN-1}) is equal to $I(t)=E$), which we found to be between $21$ and $27\,\text{Hz}$ for the parameters in this study.

From the computed phase values, we estimate the remaining available phase for the burst after the first spike as:
\begin{equation*}
    \Delta \theta = \theta_{\SL,-2} \, - \, \theta^1_{-2}
\end{equation*}
then convert the phase difference into a temporal interval using the relation $\Delta \theta = \omega \Delta t$, with $\Delta t$ is in milliseconds. Finally, by considering the intra-burst frequency $f_{burst}$ (in $\mathrm{Hz}$), and accounting for the differences in units between frequency and time variables, we estimate that $\left[ 10^{-3} \times \Delta t \right] = [\, \text{number of spikes} -1 \, ] / f_{burst}$. Then, 
\begin{equation}\label{E:numSpikesUnshifted}
\text{number of spikes} = 1 + \left[ \frac{\Delta \theta}{1000 \times \omega} \right] \times f_{burst} \,.
\end{equation}
For numerical calculations, we set $f_{burst} = 27\,\mathrm{Hz}$ as this was the predominant frequency observed in the constantly forced system.
Within formula \eqref{E:numSpikesUnshifted}, we round to the next largest integer as a fractional quantity still allows for an additional spike at the \Rev{folded} saddle.
We add one to account for the spike that occurred before $\theta^{1}_{-2}$.

We plot contours as the number of estimated spikes transitions through the integers (Figure~\ref{fig:theta_dist_MAT}). 
Observe that the estimated boundaries follow the shape of the actual count boundaries and remain in proximity to these boundaries.
The error in estimation is attributed to the limitations of this approach outlined in the beginning of this section, namely that the estimates are based off  of quantities derived in the $\varepsilon\to0$ limit. \Rev{Indeed, when we repeated the relevant computations using smaller values of $\eps$ (not shown), we observed that the discrepancy decreased accordingly.}

\begin{figure}
    \begin{center}
        \includegraphics[scale=1]{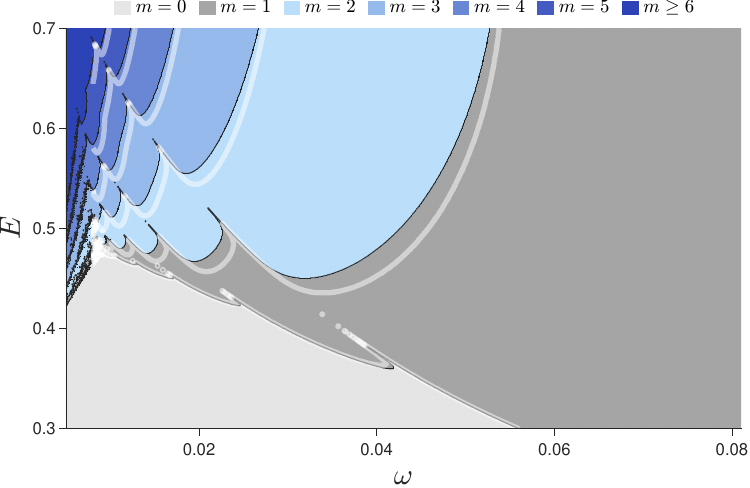}
        \caption{\textbf{Estimating spikes within a burst according to \eqref{E:numSpikesUnshifted}}. 
The transparent white curves mark where the estimated number of spikes transitions from $0\to1\to2\to\dots\to6$, when moving up and to the left in the plane (see legend).
The white curves are superimposed on the actual contours according to Figure~\ref{fig:bifurcation_E_vs_w_region_I} and only shown for $\omega > 0.008$.}
\label{fig:theta_dist_MAT}
    \end{center}
\end{figure}

\Rev{\section{Periodic forcing in the Morris-Lecar model}\label{sec:ML}

\begin{figure}
    \centering
    \includegraphics[width=\textwidth]{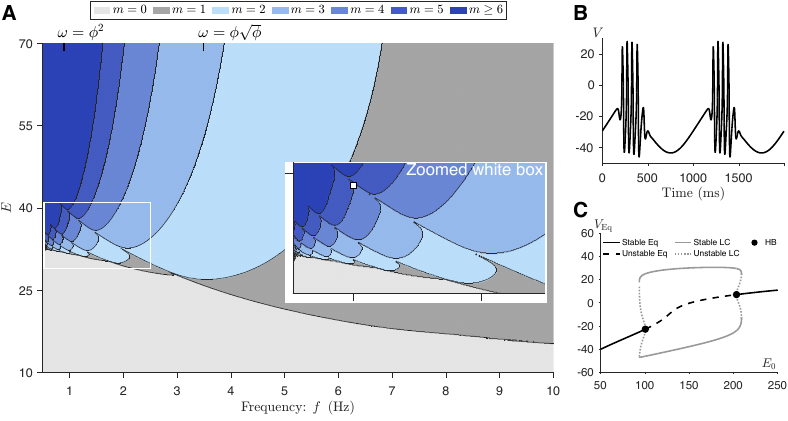}
    \caption{\Rev{\textbf{Spike-adding in the Morris-Lecar model.}
    \emph{\textbf{A}} Spike-adding transitions as forcing frequency $f$ varies from 0.5 to 10 Hz and forcing amplitude from 10 to 70 with $E_0$ fixed at 80 (see Eq.~\eqref{E:ML-input}). 
    Cusp-like transitions emerge as in the forced FHN model. 
    The white box highlights the region with a non-monotonic relationship between the number of spikes and either forcing parameter.
    We omit the portion of the gray region (1-spike or no-spikes) up to $\omega=\phi$, or $f\approx12\,\text{Hz}$.
    \emph{\textbf{B}} A sample bursting trajectory (see white scatter point in zoomed in panel of \emph{\textbf{A}}) containing putative canard-with-head segments near spike initiation and termination.
    \emph{\textbf{C}} Bifurcation diagram of the unforced Morris–Lecar model, showing two subcritical Hopf bifurcations and nearby saddle-nodes of limit cycles, similar to the unforced FHN model (see Figure 1 in \cite{davison2019mixed}).}}
    \label{fig:morris-lecar}

\end{figure}

The FHN model is a simplified dynamical model for neural dynamics which it is not based on obvious underlying biophysical principles. 
\Rev{Nevertheless, when considering a more realistic model for action potential production -- the Morris-Lecar (ML) model -- we  observed similar spike-adding transitions induced by changes in the input amplitude and/or frequency; see Figure~\ref{fig:morris-lecar}.}

The full equations for the ML model are shown in Appendix \ref{A:morris-lecar}. Briefly, we follow the model presented in \cite[Chapter 3]{ermentrout2010foundations}, which is a two-dimensional model describing the dynamics for voltage, $V$, and potassium activation, $n$:
\begin{subequations} \label{E:ML}
    \begin{align}
        C_M\dfrac{dV}{dt} &= I_{\text{app}}(t) - \gL(V-\EL) -\gK n(V-\EK)-\gCa m_{\infty}(V)(V-\ECa)\\
        \dfrac{dn}{dt}  &= \phi\left(n_{\infty}(V) - n\right) / \tau_n(V) 
\end{align}
\end{subequations}
We represent the applied current, $I_{\text{app}}$ as a function of time but choose system parameters so that the ML system with constant forcing ($I_{\text{app}}(t)\equiv E_0)$ exhibits subcritical Hopf bifurcations with nearby saddle-nodes of limit cycles; see Figure~\ref{fig:morris-lecar}C. 
This bifurcation structure closely matches the structure of the studied FHN equations with constant input \cite{davison2019mixed}.
In the case of the ML model, the first subcritical Hopf point occurs for external input $E_0\approx 100$; for the FHN model the first Hopf point occurs when the external forcing magnitude is around 0.3.
Therefore, we augment the periodic external current by shifting its baseline closer to the Hopf point for constant input,
\begin{equation}\label{E:ML-input}
   I_{\text{app}} = E_0 + E\sin(\omega t).
\end{equation}

Setting the input this way closely matches the regime for the FHN model studied in our paper. 
In simulations, we set $E_0 = 80$ and vary $E$ between 10 and 70 to maintain positive current.
We further augment the forcing frequency $\omega = (2\pi / 1000) f$, so that the parameter $f$ represents the frequency in Hz, maintaining the biophysical interpretation of the model. 
To study frequency ranges similar to the FHN model, we set the parameter $\phi = 0.075$ (close to $\varepsilon=0.08$ for FHN) and choose frequency values $f$ from $>0$ to $12\, \text{Hz}$, so that the associated parameter $\omega$ varies from $\mathcal{O}(\phi^2)$ to $\mathcal{O}(\phi)$ (see the frequency range in Figure \ref{fig:morris-lecar}A).

Like the FHN model with periodic forcing, solutions from the ML system with periodic forcing exhibit bursting patterns. 
Furthermore, some solutions contained small-amplitude oscillations near spike-initiation and termination, characteristics of the canard segments studied in the periodically forced FHN system (see Figure \ref{fig:morris-lecar}B).
In simulations that swept across forcing parameters in ranges analogous to those used in the FHN simulations, similar spike-adding boundaries emerged.
In particular, the cusp-like boundaries were present when $\omega=\mathcal{O}(\phi\sqrt{\phi})$ with the pattern terminating near $\omega=\phi^2$ in parameter space, similar to the $\mathcal{O}(\varepsilon\sqrt{\varepsilon})$ to $\mathcal{O}(\varepsilon^2)$ studied in the FHN system.
We do not perform the same detailed analysis of folded equilibria as in the FHN prototype, but we expect similar dynamics studied in the simpler FHN system to hold for the more biologically realistic Morris-Lecar model.
}

\section{Discussion}\label{Sec:discussion}

Efforts to understand the brain’s complex dynamics benefit from detailed insights into individual neurons and their interactions. Prominent examples of these complex dynamics are  bursting patterns, which are prevalent in many brain states.  These patterns, involving bursts of action potentials from neurons \textit{in vivo}, are closely associated with underlying brain rhythms such as delta, theta, alpha, beta and gamma,  that together cover frequencies between 0.5 Hz to 70 Hz (see \cite{CannontAl2014}, for a review). 

In this study, we investigated the spike-adding mechanisms in a three-dimensional neuronal model with fast, slow, and super-slow timescales, inspired by the FitzHugh-Nagumo (FHN)  model  driven by low-frequency periodic inputs\Rev{, and tested on the Morris-Lecar (ML) as well.} The FHN \Rev{and ML} model\Rev{s}, two-dimensional \Rev{reductions}
of the Hodgkin-Huxley model, effectively capture neuronal excitability and the characteristic large excursions known as spikes. 

Our analysis, based on both numerical simulations and geometric singular perturbation theory, highlighted how the number of spikes within each burst can be regulated by tuning the frequency and amplitude of the external input.   We found and characterized key bifurcations and
mechanisms such as the folded saddle and folded node canards,  that drive spike addition and shape the boundaries between regions with different spike counts. 
We identified six regions representing different behaviors in the FHN model with periodic input, based on the input’s frequency $\omega$ and amplitude $E$. A summary of the results is presented in Figure~\ref{F:ROI}. The analytically calculated boundaries between these regions  correspond to saddle-node bifurcations of folded equilibria on $\Fl$ (or $\Fr$), the left (right) knee of the critical manifold $\CM$,   or correspond to the changes in stability from \Rev{folded} node to \Rev{folded} saddle on $\Fl$ (or $\Fr$).
 We did not analyze the system's dynamics  inside Region I (where the amplitude $E$ is small) because there was  no bursting there. We focused instead  on Regions II and III (Figure~\ref{F:ROI}; $E$ of moderate values) and discussed two key spike-adding mechanisms.  
 
 One way to add spikes in the burst is  through a folded saddle canard; this situation occurred in both regions II and III.  Another mechanism, which only appears in region II,  is to create spikes through a  folded node canard.  
Furthermore, we demonstrated how these spike-adding mechanisms organize the boundaries in parameter space that separate regions with different spike counts (Figure~\ref{fig:bifurcation_E_vs_w_region_I} and Figure~\ref{F:ROI}-- right panel) by numerically calculating the $L^2$ norms of trajectories and estimating the delay-to-spike within each burst.

\Rev{We speculate that pitchfork and transcritical bifurcations of maximal canards (\cite[Figure 8.12]{kuehn2015multiScaleGSPT}), arising at specific ratios of the folded node eigenvalues, may govern the escape dynamics of solutions as they pass through the folded node. To explore this, we computed eigenvalue ratios at the folded node across the parameter space and overlaid their level sets on the spike-adding regions (figures are not shown). We observed partial alignment between low-integer ratio contours and cusp locations in the spike-adding diagrams, particularly for the first few transitions. Although these patterns became less clear at larger (biologically realistic) time-scale parameters, they support the hypothesis that canard bifurcations could influence the underlying structure of spike-adding dynamics, especially in the extreme (outside the biological relevant) multiple-timescale regime. A more systematic study of this relationship remains a promising direction for future work.}

\Rev{The bursting patterns analyzed in our paper do not fit into any of the three main known bursting types -- square, elliptic or parabolic. In particular, these bursts do not satisfy the mathematical definition of the parabolic bursting since the fast subsystem of the fast/slow/super-slow system is only one-dimensional and cannot generate spikes on its own. However, we observed in our numerical experiments (not shown here) that the bursting patterns  of the forced FHN model exhibit a \emph{parabolic-burster--like}  property. Shortly, their inter-spike intervals increase towards the end of the bursting phase, akin to those reported in classical parabolic bursting \cite{RinzelErmentrout1998}. These bursts  also show an \emph{elliptic-burst-like} feature.  Indeed, both the FHN and the ML models with constant input exhibit bistability (admittedly, in a very small parameter range), with the stable limit cycle surrounding the stable equilibrium point  e.g., see Figure \ref{fig:morris-lecar},C).   When this fast(er) dynamics is coupled with the slow dynamics of the driving force, the trajectory slowly drifts between the spiking and quiet phases like in the elliptic bursting.  Thus, while the bursts investigated in our work are neither parabolic nor elliptic, they still share some attributes with each of them.
}

\Rev{Our numerical results provide compelling evidence of a consistent structure in the spike-adding transitions and enhance the current understanding of spike-adding mechanisms.} Future research will focus on exploring the dynamics in regions IV, V, and VI of the parameter space in Figure~\ref{F:ROI}, where the amplitude $E$ is high and the system exhibits ``held-small amplitude oscillations" as defined by \cite{espanol2015BZ}. 
\Rev{We acknowledge that a full characterization of the bifurcation structure of spike-adding dynamics  is an important theoretical question that merits further investigation. This is especially true in systems with three distinct timescales and canard-induced transitions, where standard bifurcation tools may be insufficient. Developing and applying appropriate techniques to address this challenge is a natural next step in our research.}
We also plan to extend in the future this analysis to other models of neuronal excitability.

\section*{Acknowledgment}

This work was funded by NSF-RTG award 1840260 (PM), The Stanley-UI Foundation Support Organization (RC), and NSF award 2037828 and The Simons Foundation MPS--TSM--00008005 (ZA).

\section{Appendix}
\label{Sec:appendix}
\subsection{Numerical simulations for \Rev{spike-adding} diagram}\label{Subsec:appendix-simulations}
The bifurcation diagrams for spike counts (Figures~\ref{fig:bifurcation_E_vs_w},~\ref{fig:bifurcation_E_vs_w_region_I}) and $L^2$ norms (Figure~\ref{fig:L2-bifurcation}) were obtained with a parameter sweep of the forcing amplitude, $E$, and frequency, $\omega$. 
Each simulation was initialized at the equilibrium point of the unforced problem ($E=0$), $(x,y)\approx(-1.1994, -1.4993)$. 
 \Rev{We first ran the simulation for two} input periods \Rev{that were then excluded from the analysis,} to allow transient solutions to drop out. 

Spike counts were determined as the number of times the solution crossed the upper fold line, $x=1$, during the next two input periods.
The count was divided by two (to normalize according to the number of input periods) and rounded down with the floor function.
The $L^2$ norms were calculated according to \eqref{E:L2-norm} with the midpoint method for quadrature applied to simulated solutions. 

All simulations were performed in MATLAB (2023a) with the stiff differential equation solver `ode15s()' with a user-supplied Jacobian, and `RelTol' and `AbsTol' set to $10^{-8}$ and $10^{-10}$, respectively. The details of the $[\omega, E]$ parameter mesh are as follows:
\begin{itemize}
    \item  For Figure~\ref{fig:bifurcation_E_vs_w}:  $(\omega,E)\in[0.0075,1.2]\times[0.1,1]$ with a step-size in both $\omega$ and $E$ of $5\times10^{-4}$.
    \item  For Figures~\ref{fig:bifurcation_E_vs_w_region_I}~and~\ref{fig:L2-bifurcation}: we used the grid $(\omega,E)\in[0.003,0.1]\times[0.3,0.7]$ with a step-size in $\omega$ of $5\times10^{-5}$ and a step-size in $E$ of $5\times10^{-4}$.
    \item  For the zoomed-in box of Figure~\ref{fig:L2-bifurcation}:  $(\omega,E)\in[0.01875,0.026325]\times[0.455,0.525]$ with step-sizes of $6.25\times10^{-6}$ and $6.25\times10^{-5}$ for $\omega$ and $E$, respectively.
    \item \Rev{For the Morris-Lecar simulations of Figure~\ref{fig:morris-lecar}: $(f,E)\in[0.5,10]\times[10,70]$ with step-sizes of $5\times10^{-3}$ and $7.5\times10^{-2}$ for $f$ and $E$, respectively.}
\end{itemize}

\subsection{Series approximations to invariant manifolds of the \Rev{folded} saddle}\label{Sec:SeriesApproximations}
We compute series expansions for the stable and unstable manifolds of the folded saddle.
The stable eigenvector of the \Rev{folded} saddle (see Proposition~\ref{Th:eig}) is given by
\[
(u,\theta)^\top = \left(\lambda_{S\ell,1},-2\delta\right)^\top.
\]
In the $\theta \times u$ plane, this corresponds to a slope of $-\lambda_{S\ell,1}/(2\delta)$, which is finite for $\delta \neq 0$.  
So, we write the stable manifold of the \Rev{folded} saddle, $W^s_{S\ell}$, as a graph with a series expansion of $u$ in terms of $\theta$.  
But first, we translate the \Rev{folded} saddle to the origin with $\hat{\theta} = \theta - \theta_{S\ell}$, so that we can write
\begin{equation}\label{E:u_of_theta_hat}
u = a_1\hat{\theta} + a_2\hat{\theta}^2 + a_3\hat{\theta}^3 + a_4\hat{\theta}^4 + a_5\hat{\theta}^5 + \bigO(\hat{\theta}^6),
\end{equation}
with coefficients that need to be calculated. 
We first differentiate \eqref{E:u_of_theta_hat} with respect to $\hat{\theta}$,
\begin{equation}\label{E:du_dtheta_hat}
	\dfrac{du}{d\hat{\theta}} = a_1 + 2a_2\hat{\theta}+3a_3\hat{\theta}^2 + 4a_4\hat{\theta}^3 + 5a_5\hat{\theta} ^4 + \bigO(\hat{\theta}^5).
\end{equation}
Then we account for the translation of the folded saddle to the origin, replace  $\theta_{S\ell}$ and $G(0)=\mu$ according to \eqref{E:saddle} and \eqref{E:fcG}, implement the trigonometric formula $\cos (\alpha+\beta) = \cos \alpha \cos \beta - \sin \alpha \sin \beta$, and divide the equations in \eqref{E:FHN-desing}, to compute 
\begin{equation}\label{E:u_prime_by_theta_prime}
\dfrac {du}{d\hat{\theta}} = \frac{\mu \cos \hat{\theta} - \sqrt{R_{\delta}^2 - \mu^2} \sin \hat{\theta} - G(u)} {\delta u (u-2)}. 
\end{equation}

Using the series expansions $\cos \hat{\theta} = 1 - \hat{\theta}^2/2! + \hat{\theta}^4/4! + \bigO(\hat{\theta}^6)$ and $\sin \hat{\theta} = \hat{\theta} - \hat{\theta}^3/3! + \hat{\theta}^5/5! + \bigO(\hat{\theta}^7)$, and replacing  $u$ with \eqref{E:u_of_theta_hat} in the right hand-side of equation \eqref{E:u_prime_by_theta_prime}, we obtain the series
\begin{equation} \label{E:du/dtheta-b-cooeff}
\dfrac{du}{d\hat{\theta}}
 = b_0 + b_1\hat{\theta} + b_2\hat{\theta}^2+b_3\hat{\theta}^3 + b_4\hat{\theta}^4 + \bigO(\hat{\theta}^5),
\end{equation}
where coefficients $b_j$ have nonlinear dependencies on the coefficients $a_1,\dots, a_5$ and the constants $\delta$, $R_{\delta}$, and $\mu$. 
Expressions for the $b_j$ coefficients are found with the aid of \textit{Wolfram Mathematica} (see below).

The values $a_1,\dots, a_5$ are then found by solving the nonlinear system of equations defined by equating coefficients in \eqref{E:du_dtheta_hat} with those in \eqref{E:du/dtheta-b-cooeff}:
\begin{equation}\label{E:expansion_coefficient_system_of_eqs}
	\begin{aligned}
		a_1 & = b_0\left(a_1,\dots,a_5;\delta,R_\delta,\mu\right)\\
		2a_2 & = b_1\left(a_1,\dots,a_5;\delta,R_\delta,\mu\right)\\
		&\vdots\\
		5a_5 & =b_4\left(a_1,\dots,a_5;\delta,R_\delta,\mu\right)
	\end{aligned} 
\end{equation}
A solution to the nonlinear system \eqref{E:expansion_coefficient_system_of_eqs} is calculated in MATLAB with `{fsolve()}'.  

To generate a curve tangent to the stable eigenvector, we use the initial guess 
\begin{equation*}
(a_1,a_2,a_3,a_4,a_5) = \left(-\lambda_{S\ell,1}/(2\delta), 0, 0, 0, 0\right).
\end{equation*}  
To align the approximated curve with solutions of the desingularized system, we translate back to the desingularized coordinates $(u, \theta)$.

The same procedure is used to calculate the local unstable manifold of the folded saddle, $W^u_{S\ell}$, except the initial guess provided to `{fsolve()}' promotes a curve tangent to the unstable eigenvector (see Proposition~\ref{Th:eig}),
\begin{equation*}
(a_1,a_2,a_3,a_4,a_5) = \left(-\lambda_{S\ell,2}/(2\delta), 0, 0, 0, 0\right).
\end{equation*}

\subsection{\Rev{Folded} saddle invariant manifold coefficients}
We use \textit{Wolfram Mathematica} to find the coefficients in the expansion \eqref{E:du/dtheta-b-cooeff} of \eqref{E:u_prime_by_theta_prime}. In the following, we set $C=\sqrt{R_\delta^2-\mu^2}$
{\small{
\begin{align*}
	b_0 &= \frac{a_1+C}{2 a_1 \delta } \\
	b_1 & = \frac{-2 a_1^3 b+a_1^3+a_1^2 C+a_1 \mu -2 a_2 C}{4 a_1^2 \delta } \\
	b_2 & = \tfrac{1}{24a_1^3\delta}\left[-2 a_1^5 b+3 a_1^5+3 a_1^4 C-12 a_1^3 a_2 b+6 a_1^3 a_2+3 a_1^3 \mu -2 a_1^2 C + \dots\right. \notag\\
	& \quad \left. -6 a_1 a_2 \mu -12 a_1 a_3 C+12 a_2^2 C\right] \\
	b_3 & =\tfrac{1}{48 a_1^4 \delta }\left[-2 a_1^7 b+3 a_1^7+3 a_1^6 C-8 a_1^5 a_2 b+12 a_1^5 a_2+3 a_1^5 \mu + \dots\right.\notag\\
	& \quad\left.+6 a_1^4 a_2 C-24 a_1^4 a_3 b+12 a_1^4 a_3-2 a_1^4 C-a_1^3 \mu +4 a_1^2 a_2 C+ \dots\right.\notag \\
	&\quad\left. -12 a_1^2 a_3 \mu -24 a_1^2 a_4 C+12 a_1 a_2^2 \mu +48 a_1 a_2 a_3 C-24 a_2^3 C\right] \\
	b_4 & = \tfrac{1}{480 a_1^5 \delta }\left[-10 a_1^9 b+15 a_1^9+15 a_1^8 C-60 a_1^7 a_2 b+90 a_1^7 a_2+15 a_1^7 \mu +\dots\right. \notag \\
	&\quad\left. 60 a_1^6 a_2 C-80 a_1^6 a_3 b+120 a_1^6 a_3-10 a_1^6 C-40 a_1^5 a_2^2 b+60 a_1^5 a_2^2+\dots\right.\notag \\
	&\quad \left.30 a_1^5 a_2 \mu +60 a_1^5 a_3 C-240 a_1^5 a_4 b+120 a_1^5 a_4-5 a_1^5 \mu +2 a_1^4 C+\dots\right.\notag \\
	&\quad \left.10 a_1^3 a_2 \mu +40 a_1^3 a_3 C-120 a_1^3 a_4 \mu -240 a_1^3 a_5 C + \dots\right.\notag \\
	&\quad\left.-40 a_1^2 a_2^2 C+240 a_1^2 a_2 a_3 \mu +480 a_1^2 a_2 a_4 C+240 a_1^2 a_3^2 C+\dots\right.\notag \\
	&\quad\left.-120 a_1 a_2^3 \mu -720 a_1 a_2^2 a_3 C+240 a_2^4 C\right]
\end{align*}
}}

The desired coefficients $a_1,\dots,a_5$ are found by solving \eqref{E:expansion_coefficient_system_of_eqs}.  
For example, the first coefficient for the stable manifold is found by setting
\begin{equation*}
	\frac{a_1+\sqrt{R_{\delta}^2-\mu ^2}}{2 a_1 \delta } = a_1
\end{equation*}
with solution given by
\begin{align*}
	a_1 &= \frac{1 \pm \sqrt{1 + 8\delta\sqrt{R_{\delta}^2-\mu^2}}}{4\delta} 
	 	  = -\frac{1}{2\delta}\cdot\frac{1}{2} \left(-1 \mp\sqrt{1+8\delta\sqrt{R_{\delta}^2-\mu^2}}\right).
\end{align*}
Choosing the subtraction term  gives exactly $a_1 = -\lambda_{S\ell,1} / (2\delta)$ as expected for the stable manifold. The other solution gives $a_1 = -\lambda_{S\ell,2} / (2\delta)$ as expected for the unstable manifold (see Proposition~\ref{Th:eig}). 

Similarly, with the help of \textit{Wolfram Mathematica}, we can write an equation for the second expansion coefficient of the stable manifold,
\[
	a_2 = \frac{\lambda_{S\ell,1}^3(2b-1) + 2 \delta \lambda_{S\ell,1}^2 \sqrt{R_{\delta}^2-\mu^2} - 4 \mu \delta^2 \lambda_{S\ell,1}}{16\delta^2\left(\lambda_{S\ell,1}^2 + \delta \sqrt{R_{\delta}^2-\mu^2}\right)}.
\]
A similar result is obtained for the unstable manifold by replacing $\lambda_{S\ell,1}$ with $\lambda_{S\ell,2}$ in the equation above.

Rather than obtaining exact coefficients (as shown for $a_1$ and $a_2$ here), we relied on numerical solutions from {`fsolve()'} in MATLAB, as described in the previous section.

\Rev{\subsection{Morris-Lecar equations and parameters}\label{A:morris-lecar}
The system of differential Equations in Eq.~\eqref{E:ML} is supplemented by the steady-state equations,
\begin{align*}
    m_\infty(V)&=\frac{1}{2}\left[1+\tanh((V-V_1)/V_2)\right] \\
    {\tau_n}(V)&=1/\cosh((V-V_3)/(2V_4)),
\end{align*}
and voltage-dependent time constant for potassium activation
\begin{equation*}
    {n_\infty}(V)=\frac{1}{2}\left[1+\tanh((V-V_3)/V4)\right].
\end{equation*}
The parameters were chosen to induce subcritical Hopf bifurcations and are given in Table~\ref{T:ML-pars}.
\begin{table}
\caption{Parameters used (see Eq.~\eqref{E:ML}) for Morris-Lecar simulations with constant and periodic input.}\label{T:ML-pars}
\begin{center}
    \begin{tabular}{c  c c c c c c c c c c c c}
    \hline
        Parameter & $\phi$ & $\gCa$ & $V_3$ & $V_4$ & $\ECa$ & $\EK$ & $\EL$ & $\gK$ & $\gL$ & $V_1$ & $V_2$ & $C_M$ \\
        Value & 0.075 & 4.4 & 2 & 30 & 120 & -84 & -60 & 8 & 2 & -1.2 & 18 & 20 \\
    \hline
    \end{tabular}
\end{center}
\end{table}
}


\begin{thebibliography}{10}
\providecommand{\url}[1]{#1}
\csname url@samestyle\endcsname
\providecommand{\newblock}{\relax}
\providecommand{\bibinfo}[2]{#2}
\providecommand{\BIBentrySTDinterwordspacing}{\spaceskip=0pt\relax}
\providecommand{\BIBentryALTinterwordstretchfactor}{4}
\providecommand{\BIBentryALTinterwordspacing}{\spaceskip=\fontdimen2\font plus
\BIBentryALTinterwordstretchfactor\fontdimen3\font minus
  \fontdimen4\font\relax}
\providecommand{\BIBforeignlanguage}[2]{{%
\expandafter\ifx\csname l@#1\endcsname\relax
\typeout{** WARNING: IEEEtran.bst: No hyphenation pattern has been}%
\typeout{** loaded for the language `#1'. Using the pattern for}%
\typeout{** the default language instead.}%
\else
\language=\csname l@#1\endcsname
\fi
#2}}
\providecommand{\BIBdecl}{\relax}
\BIBdecl

\bibitem{Rinzel1985Bursting}
J.~Rinzel, ``Bursting oscillations in an excitable membrane model,'' in
  \emph{Ordinary and Partial Differential Equations}, B.~D. Sleeman and R.~J.
  Jarvis, Eds.\hskip 1em plus 0.5em minus 0.4em\relax Berlin, Heidelberg:
  Springer Berlin Heidelberg, 1985, pp. 304--316.

\bibitem{Rinzel1987Formal}
------, \emph{A Formal Classification of Bursting Mechanisms in Excitable
  Systems}.\hskip 1em plus 0.5em minus 0.4em\relax Berlin, Heidelberg: Springer
  Berlin Heidelberg, 1987, pp. 267--281.

\bibitem{RinzelErmentrout1998}
J.~Rinzel and B.~Ermentrout, ``Analysis of neural excitability and
  oscillations,'' in \emph{Methods in Neuronal Modeling: From Ions to Networks,
  second edition}, C.~Koch and I.~Segev, Eds.\hskip 1em plus 0.5em minus
  0.4em\relax London, England: The MIT Press, 1985, pp. 251--291.

\bibitem{Izhikevich2006Dynamical}
\BIBentryALTinterwordspacing
E.~M. Izhikevich, \emph{{Dynamical Systems in Neuroscience: The Geometry of
  Excitability and Bursting}}.\hskip 1em plus 0.5em minus 0.4em\relax The MIT
  Press, 07 2006. [Online]. Available:
  \url{https://doi.org/10.7551/mitpress/2526.001.0001}
\BIBentrySTDinterwordspacing

\bibitem{ErmentroutKopell1986}
B.~Ermentrout and N.~Kopell, ``Parabolic bursting in an excitable system
  coupled with a slow oscillation,'' \emph{SIAM J. Appl. Math.}, vol.~46, pp.
  233--253, 1986.

\bibitem{FoxRotsteinNadim2014}
D.~M. Fox, H.~G. Rotstein, and F.~Nadim, ``Bursting in neurons and small
  networks,'' in \emph{Encyclopedia of Computational Neuroscience}, D.~Jaeger
  and R.~Jung, Eds.\hskip 1em plus 0.5em minus 0.4em\relax New York, USA:
  Springer Science+Business Media, 2014, pp. 1--17.

\bibitem{Innocenti2007Dynamical}
\BIBentryALTinterwordspacing
G.~Innocenti, A.~Morelli, R.~Genesio, and A.~Torcini, ``{Dynamical phases of
  the Hindmarsh-Rose neuronal model: Studies of the transition from bursting to
  spiking chaos},'' \emph{Chaos: An Interdisciplinary Journal of Nonlinear
  Science}, vol.~17, no.~4, p. 043128, 12 2007. [Online]. Available:
  \url{https://doi.org/10.1063/1.2818153}
\BIBentrySTDinterwordspacing

\bibitem{Nowacki2012Dynamical}
\BIBentryALTinterwordspacing
J.~Nowacki, H.~M. Osinga, and K.~Tsaneva-Atanasova, ``Dynamical systems
  analysis of spike-adding mechanisms in transient bursts,'' \emph{The Journal
  of Mathematical Neuroscience}, vol.~2, 2012. [Online]. Available:
  \url{https://doi.org/10.1186/2190-8567-2-7}
\BIBentrySTDinterwordspacing

\bibitem{Sherman2001}
S.~M. Sherman, ``Tonic and burst firing: dual modes of thalamocortical relay,''
  \emph{Trends in Neuroscience}, vol.~24, pp. 122--126, Feb 2001.

\bibitem{Terman1992}
D.~Terman, ``The transition from bursting to continuous spiking in excitable
  membrane models,'' \emph{J. Nonlinear Sci.}, vol.~2, pp. 135--182, 1992.

\bibitem{XJWang1993}
X.-J. Wang, ``Genesis of bursting oscillations in the {H}indmarsh-{R}ose model
  and homoclinicity to a chaotic saddle,'' \emph{Physica D}, vol.~62, pp.
  263--274, 1993.

\bibitem{HanBi2012}
\BIBentryALTinterwordspacing
X.~Han and Q.~Bi, ``Slow passage through canard explosion and mixed-mode
  oscillations in the forced van der pol's equation,'' \emph{Nonlinear
  Dynamics}, vol.~68, no.~3, pp. 275--283, 2012. [Online]. Available:
  \url{https://doi.org/10.1007/s11071-011-0226-9}
\BIBentrySTDinterwordspacing

\bibitem{Ma2022}
\BIBentryALTinterwordspacing
X.~Ma, Q.~Bi, and L.~Wang, ``Complex periodic bursting structures in the
  rayleigh–van der pol–duffing oscillator,'' \emph{Journal of Nonlinear
  Science}, vol.~32, no.~1, p.~25, 2022. [Online]. Available:
  \url{https://doi.org/10.1007/s00332-022-09781-1}
\BIBentrySTDinterwordspacing

\bibitem{Barrio2024Dynamics}
R.~Barrio, S.~Ibanez, J.~A. Jover-Galtier, A.~Lozano, M.~A. Martinez,
  A.~Mayora-Cebollero, C.~Mayora-Cebollero, L.~Perez, S.~Serrano, and
  R.~Vigara, ``Dynamics of excitable cells: spike-adding phenomena in action,''
  \emph{SeMA Journal}, vol.~81, pp. 113--146, 2024.

\bibitem{Barrio2020Spike-adding}
\BIBentryALTinterwordspacing
R.~Barrio, S.~Ibáñez, L.~Pérez, and S.~Serrano, ``Spike-adding structure in
  fold/hom bursters,'' \emph{Communications in Nonlinear Science and Numerical
  Simulation}, vol.~83, p. 105100, 2020. [Online]. Available:
  \url{https://www.sciencedirect.com/science/article/pii/S1007570419304198}
\BIBentrySTDinterwordspacing

\bibitem{Barrio2021Classification}
R.~Barrio, S.~Ibanez, L.~Perez, and S.~Serrano, ``Classification of fold/hom
  and fold/hopf spike-adding phenomena,'' \emph{Chaos: An Interdisciplinary
  Journal of Nonlinear Science}, vol.~31, no.~4, p. 043120, 04 2021.

\bibitem{Channell2007Origin}
\BIBentryALTinterwordspacing
P.~Channell, G.~Cymbalyuk, and A.~Shilnikov, ``Origin of bursting through
  homoclinic spike adding in a neuron model,'' \emph{Phys. Rev. Lett.},
  vol.~98, p. 134101, Mar 2007. [Online]. Available:
  \url{https://link.aps.org/doi/10.1103/PhysRevLett.98.134101}
\BIBentrySTDinterwordspacing

\bibitem{Linaro2012Codimension}
\BIBentryALTinterwordspacing
D.~Linaro, A.~Champneys, M.~Desroches, and M.~Storace, ``Codimension-two
  homoclinic bifurcations underlying spike adding in the hindmarsh--rose
  burster,'' \emph{SIAM Journal on Applied Dynamical Systems}, vol.~11, no.~3,
  pp. 939--962, 2012. [Online]. Available:
  \url{https://doi.org/10.1137/110848931}
\BIBentrySTDinterwordspacing

\bibitem{Carter2020Spike-Adding}
\BIBentryALTinterwordspacing
P.~Carter, ``Spike-adding canard explosion in a class of square-wave
  bursters,'' \emph{Journal of Nonlinear Science}, vol.~30, p. 2613–2669,
  2020. [Online]. Available: \url{https://doi.org/10.1007/s00332-020-09631-y}
\BIBentrySTDinterwordspacing

\bibitem{Penalva2024Dynamics}
\BIBentryALTinterwordspacing
J.~Penalva, M.~Desroches, A.~E. Teruel, and C.~Vich, ``Dynamics of a
  piecewise-linear {M}orris–{L}ecar model: Bifurcations and spike adding,''
  \emph{Journal of Nonlinear Science}, vol.~34, 2024. [Online]. Available:
  \url{https://doi.org/10.1007/s00332-024-10029-3}
\BIBentrySTDinterwordspacing

\bibitem{DESROCHES2016Spike-adding}
\BIBentryALTinterwordspacing
M.~Desroches, M.~Krupa, and S.~Rodrigues, ``Spike-adding in parabolic bursters:
  The role of folded-saddle canards,'' \emph{Physica D: Nonlinear Phenomena},
  vol. 331, pp. 58--70, 2016. [Online]. Available:
  \url{https://www.sciencedirect.com/science/article/pii/S0167278916300471}
\BIBentrySTDinterwordspacing

\bibitem{ermentrout2010foundations}
G.~B. Ermentrout and D.~H. Terman, \emph{Mathematical Foundations of
  Neuroscience}, S.~Antman, J.~Marsden, L.~Sirovich, and S.~Wiggins, Eds.\hskip
  1em plus 0.5em minus 0.4em\relax New York, NY: Springer, 2010, vol.~35.

\bibitem{desroches2018SIADS}
M.~Desroches and V.~Kirk, ``Spike-adding in a canonical three-time-scale model:
  Superslow explosion and folded-saddle canards,'' \emph{SIAM J. Applied
  Dynamical Systems}, vol.~17, no.~3, pp. 1989--2017, 2018.

\bibitem{LetsonEtAl2017}
B.~Letson, J.~E. Rubin, and T.~Vo, ``Analysis of interacting local oscillation
  mechanisms in three-timescale systems,'' \emph{SIAM Journal on Applied
  Mathematics}, vol.~77, no.~3, pp. 1020--1046, 2017.

\bibitem{buzsakiWatson2012}
G.~Buzsaki and B.~Watson, ``Brain rhythms and neural syntax: implications for
  efficient coding of cognitive content and neuropsychiatric disease,''
  \emph{Dialogues in Clinical Neuroscience}, vol.~14, no.~4, pp. 345--367,
  2012.

\bibitem{TsodyksEtAl1996}
M.~V. Tsodyks, W.~E. Skaggs, T.~J. Sejnowski, and B.~L. McNaughton,
  ``Population dynamics and theta rhythm phase precession of hippocampal place
  cell firing: A spiking neuron model,'' \emph{Hippocampus}, vol.~6, pp.
  271--280, 1996.

\bibitem{AdamEtAl2022}
E.~M. Adam, E.~N. Brown, N.~Kopell, and M.~M. McCarthy, ``Deep brain
  stimulation in the subthalamic nucleus for {P}arkinson’s disease can
  restore dynamics of striatal networks,'' \emph{PNAS}, vol. 119, pp.
  e2\,120\,808\,119:1--10, 2022.

\bibitem{YokoiEtAl2019}
R.~Yokoi, M.~Okabe, N.~Matsuda, A.~Odawara, A.~Karashima, and I.~Suzuki,
  ``Impact of sleep–wake-associated neuromodulators and repetitive
  low-frequency stimulation on human i{PSC}-derived neurons,'' \emph{Front.
  Neurosci.}, vol.~13, pp. 554:1--15, 2019.

\bibitem{davison2019mixed}
\BIBentryALTinterwordspacing
E.~N. Davison, Z.~Aminzare, B.~Dey, and N.~Ehrich~Leonard, ``Mixed mode
  oscillations and phase locking in coupled {F}itzhugh-{N}agumo model
  neurons,'' \emph{Chaos: An Interdisciplinary Journal of Nonlinear Science},
  vol.~29, no.~3, p. 033105, 2019. [Online]. Available:
  \url{https://doi.org/10.1063/1.5050178}
\BIBentrySTDinterwordspacing

\bibitem{fenichel1979GSPT}
\BIBentryALTinterwordspacing
N.~Fenichel, ``Geometric singular perturbation theory for ordinary differential
  equations,'' \emph{Journal of Differential Equations}, vol.~31, no.~1, pp.
  53--98, 1979. [Online]. Available:
  \url{https://www.sciencedirect.com/science/article/pii/0022039679901529}
\BIBentrySTDinterwordspacing

\bibitem{guckenheimer2003forced}
J.~Guckenheimer, K.~Hoffman, and W.~Weckesser, ``The forced van der {P}ol
  equation {I}: The slow flow and its bifurcations,'' \emph{SIAM Journal on
  Applied Dynamical Systems}, vol.~2, no.~1, pp. 1--35, 2003.

\bibitem{kuehn2015multiScaleGSPT}
C.~Kuehn, \emph{Multiple Time Scale Dynamics}.\hskip 1em plus 0.5em minus
  0.4em\relax Switzerland: Springer Cham, 2015, vol. 191.

\bibitem{brons2006MMO}
M.~Br{\o}ns, M.~Krupa, and M.~Wechselberger, \emph{Mixed Mode Oscillations due
  to the Generalized Canard Phenomenon}, ser. Fields Insititute
  Communications.\hskip 1em plus 0.5em minus 0.4em\relax American Mathematical
  Society, 2006, no.~49, pp. 39--64.

\bibitem{CurtuRubin2011}
R.~Curtu and J.~Rubin, ``Interaction of canard and singular {H}opf mechanisms
  in a neural model,'' \emph{SIAM J. Appl. Dyn. Syst.}, vol.~10, no.~4, pp.
  1443--1479, 2011.

\bibitem{desroches2010numerical}
\BIBentryALTinterwordspacing
M.~Desroches, B.~Krauskopf, and H.~M. Osinga, ``Numerical continuation of
  canard orbits in slow–fast dynamical systems,'' \emph{Nonlinearity},
  vol.~23, no.~3, p. 739, feb 2010. [Online]. Available:
  \url{https://dx.doi.org/10.1088/0951-7715/23/3/017}
\BIBentrySTDinterwordspacing

\bibitem{szmolyan2001R3}
P.~Szmolyan and M.~Wechselberger, ``Canards in {$\mathbb{R}^3$},''
  \emph{Journal of Differential Equations}, vol. 177, no.~2, pp. 419--453,
  2001.

\bibitem{wechselberger2005bifurcation}
M.~Wechselberger, ``Existence and bifurcation of canards in {$\mathbb{R}^3$} in
  the case of a folded node,'' \emph{SIAM Journal on Applied Dynamical
  Systems}, vol.~4, no.~1, pp. 101--139, 2005.

\bibitem{wechselberger2013nonauto}
M.~Wechselberger, J.~Mitry, and J.~Rinzel, \emph{Canard Theory and
  Excitability}.\hskip 1em plus 0.5em minus 0.4em\relax Cham: Springer
  International Publishing, 2013, pp. 89--132.

\bibitem{buzsaki2004neuronal}
G.~Buzs{\'a}ki and A.~Draguhn, ``Neuronal oscillations in cortical networks,''
  \emph{Science}, vol. 304, no. 5679, pp. 1926--1929, 2004.

\bibitem{buzsaki2012rhythms}
G.~Buzs{\'a}ki, C.~A. Anastassiou, and C.~Koch, ``The origin of extracellular
  fields and currents—eeg, ecog, lfp and spikes,'' \emph{Nature Reviews
  Neuroscience}, vol.~13, no.~6, pp. 407--420, 2012.

\bibitem{herweg2020theta}
N.~A. Herweg, E.~A. Solomon, and M.~J. Kahana, ``Theta oscillations in human
  memory,'' \emph{Trends in Cognitive Sciences}, vol.~24, no.~3, pp. 208--227,
  2020.

\bibitem{cho2023gamma}
K.~H. Cho, H.~J. Lim, and J.~Jeong, ``Gamma rhythm as a guardian of brain
  health,'' \emph{eLife}, vol.~12, p. e100238, 2023.

\bibitem{canolty2006high}
R.~T. Canolty, E.~Edwards, S.~S. Dalal, M.~Soltani, S.~S. Nagarajan, H.~E.
  Kirsch, M.~S. Berger, N.~M. Barbaro, and R.~T. Knight, ``High gamma power is
  phase-locked to theta oscillations in human neocortex,'' \emph{Science}, vol.
  313, no. 5793, pp. 1626--1628, 2006.

\bibitem{basar2013brain}
E.~Ba{\c{s}}ar and B.~G{\"u}ntekin, ``Brain's alpha, beta, gamma, delta, and
  theta oscillations in neuropsychiatric diseases: proposal for biomarker
  strategies,'' \emph{Supplement to Clinical Neurophysiology}, vol.~62, pp.
  19--54, 2013.

\bibitem{IzhikevichFitzHugh2006}
E.~M. Izhikevich and R.~FitzHugh, ``{F}itzhugh-{N}agumo model,''
  \emph{Scholarpedia}, vol. 1(9), p. 1349, 2006.

\bibitem{Bold2003forced}
K.~Bold, C.~Edwards, J.~Guckenheimer, S.~Guharay, K.~Hoffman, J.~Hubbard,
  R.~Oliva, and W.~Weckesser, ``The forced van der pol equation ii: Canards in
  the reduced system,'' \emph{SIAM Journal on Applied Dynamical Systems},
  vol.~2, no.~4, pp. 570--608, 2003.

\bibitem{Burke2016Canards}
J.~Burke, M.~Desroches, A.~Granados, T.~J. Kaper, M.~Krupa, and T.~Vo, ``From
  canards of folded singularities to torus canards in a forced van der{P}ol
  equation,'' \emph{Journal of Nonlinear Science}, vol.~26, no.~2, pp.
  405--451, 2016.

\bibitem{kaklamanos2022bifurcations}
P.~Kaklamanos, N.~Popović, and K.~U. Kristiansen, ``Bifurcations of mixed-mode
  oscillations in three-timescale systems: An extended prototypical example,''
  \emph{Chaos: An Interdisciplinary Journal of Nonlinear Science}, vol.~32,
  no.~1, p. 013108, 2022.

\bibitem{neishtadt1987persistence}
A.~I. Neishtadt, ``Persistence of stability loss for dynamical bifurcations
  i,'' \emph{Differential Equations}, vol.~23, no.~12, pp. 1385--1391, 1987.

\bibitem{neishtadt1988persistence}
------, ``Persistence of stability loss for dynamical bifurcations ii,''
  \emph{Differential Equations}, vol.~24, no.~2, pp. 171--176, 1988.

\bibitem{hayes2016geometric}
M.~Hayes, T.~J. Kaper, P.~Szmolyan, and M.~Wechselberger, ``Geometric
  desingularization of degenerate singularities in the presence of fast
  rotation: A new proof of known results for slow passage through hopf
  bifurcations,'' \emph{Indagationes Mathematicae}, vol.~27, no.~5, pp.
  1184--1203, 2016.

\bibitem{espanol2015BZ}
M.~I. Espanol and H.~G. Rotstein, ``Complex mixed-mode oscillatory patterns in
  a periodically forced excitable belousov-zhabotinsky reaction model,''
  \emph{Chaos: An Interdisciplinary Journal of Nonlinear Science}, vol.~25,
  no.~6, p. 064612, 06 2015.

\bibitem{wieczorek2011compost}
S.~Wieczorek, P.~Ashwin, C.~M. Luke, and P.~M. Cox, ``Excitability in ramped
  systems: the compost-bomb instability,'' \emph{Proceedings of the Royal
  Society A: Mathematical, Physical and Engineering Sciences}, vol. 467, no.
  2129, pp. 1243--1269, 2011.

\bibitem{CannontAl2014}
J.~Cannon, M.~M. McCarthy, S.~Lee, J.~Lee, C.~B\"{o}rgers, M.~A. Whittington,
  and N.~Kopell, ``Neurosystems: brain rhythms and cognitive processing,''
  \emph{Eur J Neurosci.}, vol.~39, pp. 705--719, 2014.

\end{thebibliography}
\end{document}